\documentclass[12pt]{article}
\usepackage{fancyhdr, array,calc,graphicx,url,tabularx}
\usepackage{geometry}
\usepackage{latexsym}
\usepackage{amssymb}
\usepackage{amsmath}
\usepackage{makeidx}
\usepackage{enumerate}
\usepackage{verbatim}
\usepackage{tikz}
\usepackage[colorlinks=true,linkcolor=blue]{hyperref}

\usepackage{changepage}

\makeindex

{
   \end{minipage}
   \vspace*{\stretch{3}}
   \clearpage
}

\renewcommand{\gg}{\gamma}

\newcommand{\bG}{\mathbf{\Gamma}}
\newcommand{\bS}{\mathbf{\Sigma}}
\newcommand{\bd}{\mathbf{\delta}}
\newcommand{\bD}{\mathbf{\Delta}}

\newcommand{\bR}{{\mathbb{R}}}

\newcommand{\rest}{\restriction}

\newcommand{\card}[1]{{\vert #1 \vert} }

\newcommand{\forces}{\Vdash}

\renewcommand{\models}{\vDash}
\newcommand{\powerset}{{\wp}}
\newcommand{\dom}{{\rm dom}}
\newcommand{\rge}{{\rm rge}}
\newcommand{\ind}{{\rm ind}}

\newcommand{\cp}{{\rm crit }}

\newcommand{\lh}{{\rm lh}}

\newtheorem{theorem}{Theorem}[section]
\newtheorem{proposition}[theorem]{Proposition}
\newtheorem{definition}[theorem]{Definition}

\newtheorem{lemma}[theorem]{Lemma}
\newtheorem{corollary}[theorem]{Corollary}

\newtheorem{conjecture}[theorem]{Conjecture}

\newtheorem{sublemma}[theorem]{Sublemma}

\newtheorem{review}[theorem]{Review}
\numberwithin{figure}{section}

\newenvironment{proof}{{\it{
Proof.}}}{\nopagebreak\mbox{}{\hfill$\square$}
\par\bigskip}
\newenvironment{Claim}{\noindent{\it
Claim} {}}{}
\newenvironment{Case}{\noindent{\it
Case}}{}

\newcommand{\rcon}[1]{Conjecture~\ref{#1}}

\newcommand{\rprop}[1]{Proposition~\ref{#1}}
\newcommand{\rthm}[1]{Theorem~\ref{#1}}
\newcommand{\rlem}[1]{Lemma~\ref{#1}}
\newcommand{\rsublem}[1]{Sublemma~\ref{#1}}

\newcommand{\rcor}[1]{Corollary~\ref{#1}}
\newcommand{\rdef}[1]{Definition~\ref{#1}}

\newcommand{\rsubsec}[1]{Section~\ref{#1}}

\newcommand{\rsec}[1]{Section~\ref{#1}}

\newcommand{\rrev}[1]{Review~\ref{#1}}

\def\inseg{\trianglelefteq}

\def\k{\kappa}
\def\a{\alpha}
\def\b{\beta}
\def\d{\delta}
\def\s{\sigma}

\def\l{\lambda}

\def\c{{\sf{cop}}}

\def\P{{\mathcal{P} }}
\def\W{{\mathcal{W} }}
\def\Q{{\mathcal{ Q}}}
\def\mH{{\mathcal{ H}}}
\def\K{{\mathcal{ K}}}

\def\R{{\mathcal R}}
\def\X{{\mathbb X}}
\def\H{{\rm{HOD}}}
\def\M{{\mathcal{M}}}
\def\N{{\mathcal{N}}}

\def\T {{\mathcal{T}}}
\def\U{{\mathcal{U}}}
\def\S{{\mathcal{S}}}
\def\V{{\mathcal{V}}}
\def\X{{\mathcal{X}}}
\def\Y{{\mathcal{Y}}}

\def\card#1{\left|#1\right|}

\def\iff{\mathrel{\leftrightarrow}}

\def\and{\mathrel{\kern1pt\&\kern1pt}}

\def\inseg{\triangleleft}
\def\insegeq{\trianglelefteq}

\def\<#1>{\langle\,#1\,\rangle}

 \input xy
 \xyoption{all}
       
\title{Hjorth's reflection argument}

\author{Grigor Sargsyan\\
Institute of Mathematics\\
Polish Academy of Sciences\\
impan.pl/$\sim$gsargsyan}

\date{\today}

\begin{document}
\maketitle
\begin{abstract} In \cite{Hjorth}, Hjorth, assuming ${\sf{AD+ZF+DC}}$, showed that there is no sequence of length $\omega_2$ consisting of distinct $\bS^1_2$-sets. We show that the same theory implies that for $n\geq 0$, there is no sequence of length $\bd^1_{2n+2}$ consisting of distinct $\bS^1_{2n+2}$ sets. The theorem settles Question 30.21 of \cite{Kanamori}, which was also conjectured by Kechris in \cite{Kechris} (see Conjecture in Chapter 4 of \cite{Kechris} and the last paragraph of Chapter 4 of \cite{Kechris}). 
\end{abstract}

A central theme in descriptive set theory is the study of the complexity of various natural processes in terms of their ordinal lengths. In this line of thought, it is often shown that some ordinal is unreachable via processes of certain complexity. For example, there is no analytic well-founded relation of length $\omega_1$, and so $\omega_1$ is inaccessible with respect to analytic surjections with domain $\bR$. 

One way to study the complexity of a definability class $\bG$ is to seek definable ways of assigning sets from $\bG$ to ordinals in such a way that no two ordinals are assigned to the same set. For example, if $\a$ is a countable ordinal then we can assign to $\a$ the set $A_\a\subseteq \bR$ consisting of those reals that code $\a$ in some natural way.  Each $A_\a$ is Borel and clearly the assignment $\a\mapsto A_\a$ is definable. However, a remarkable theorem of Harrington (see \cite[Theorem 4.5]{Harrington}) says that such an assignment cannot exist if we further demand that each set comes from a specific Borel class.

Below ${\sf{AD}}$ is the ${\sf{Axiom\ of\ Determinacy}}$, ${\sf{ZF}}$ are the axioms of the Zermelo-Fraenkel set theory (which does not include the ${\sf{Axiom\ of\ Choice}})$ and ${\sf{DC}}$ is the axiom of dependent choice. 
\begin{theorem}[Harrington] Assume ${\sf{AD+ZF+DC}}$. If $\b<\omega_1$ then there is no injection $f:\omega_1\rightarrow \mathbf{\Pi}^0_\b$. 
\end{theorem}

Notice that the content of Harrington's theorem isn't that $\card{\mathbf{\Pi}^0_\b}<\omega_1$, which is in fact not true as $\mathbf{\Pi}^0_\b$ has continuum many distinct sets. The content of Harrington's theorem is that if we fix $\b<\omega_1$ and devise an algorithm that picks a new set from the Borel class $\mathbf{\Pi}^0_\b$ then at some stage less than $\omega_1$ our algorithm will stop outputting anything\footnote{Harrington's Theorem can also be viewed as an effective cardinality theorem. This type of questions have been investigated in \cite{Hjorth_2002} and \cite{ANDHJONEE}.}. 

Harrington's theorem was recently used by Marks and Day to prove the decomposability conjecture (see \cite{Marks}). 
\begin{definition}
Suppose $\bG\subseteq \powerset(\bR)$. We say $\k$ is $\bG$-reachable if there is an injection $f:\kappa\rightarrow \bG$, and that $\k$ is $\bG$-unreachable if it is not $\bG$-reachable. 
\end{definition}
 Let $\Theta$ be the least ordinal that is not a surjective image of $\bR$. Then in $L(\bR)$, $\Theta$ is $\powerset(\bR)$-reachable. Indeed, in $L(\bR)$, if $\a<\Theta$ then there is a pre-well-ordering (pwo\footnote{$\leq^*$ is a pwo if $\leq^*$ is a well-founded relation such that for every $x, y\in \dom(\leq^*)$, either $x\leq^* y$ or $y\leq^* x$.}) $\leq^*$ of $\bR$ that is ordinal definable and has length $\a$. We can then let $\leq^*_\a$ be the least ordinal definable pwo of $\bR$ that has length $\a$. Thus we have assigned an ordinal definable pwo of length $\a$ to each $\a$ in an ordinal definable manner. 

The above construction has a well-known Harrington-type analogue. Assume ${\sf{AD}}$. If $\b<\Theta$ and $\bG$ consists of those sets of reals whose Wadge rank\footnote{E.g. \cite[Theorem 29.16]{Kanamori} or \cite[Chapter 2.3]{Jackson}.} is $\leq \b$ then $\Theta$ is $\bG$-unreachable. This is because Wadge's lemma\footnote{See \cite[Lemma 29.15]{Kanamori}.} implies that there is a surjection $f:\bR\rightarrow \bG$, so if $\Theta$ was $\bG$-reachable then we could find a surjection $g:\bR\rightarrow \Theta$. 

Perhaps the most natural way of showing that $\k$ is $\bG$-reachable is to find a surjection $f: \bR\rightarrow \k$ such that for each $\a<\k$, $A_\a=\{ x\in \bR: f(x)<\a\}\in \bG$. In this case, the sets $(A_\a: \a<\k)$ form a strictly $\subset$-increasing sequence. Thus, the fact that $\k$ is not $\bG$-reachable via a strictly $\subset$-increasing sequence implies that $\kappa$ is inaccessible with respect to $\bG$-surjections. 

An equivalence relation on $\bR$ is called \textit{thin} if it does not have a perfect set of inequivalent elements. Assuming ${\sf{AD}}$ holds in $L(\bR)$, or in fact just ${\sf{AD^+}}$, Woodin (see \cite{CaicedoKetchersid},  \cite{ChanDet} and \cite[Theorem 0.3]{Dichotomy}), generalizing Harrington's earlier result on $\Pi^1_2$-equivalence relations, showed that if $E\subseteq \bR^2$ is a thin equivalence relation then the set $\{[x]_E: x\in \bR\}$ is well-orderable. Thus, assuming ${\sf{AD^+}}$, if $\k$  is $\bG$-unreachable than any thin equivalence relation $E\in \powerset(\bR^2)\cap \bG$ has $<\k$ many equivalence classes. 

 Below $\mathbf{\delta}^1_{n}$ is the supremum of the lengths of $\mathbf{\Delta}^1_n$ pre-well-orderings of $\bR$. In a seminal work, Jackson, building on an early work of Kechris, Kunen, Martin, Moschovakis and Solovay computed $\mathbf{\delta}^1_n$. Before Jackson's work, it was known that assuming ${\sf{AD}}^{L(\bR)}$, $\delta^1_1=\aleph_1$, $\delta^1_2=\aleph_2^{L(\bR)}$, 
 $\delta^1_3=\aleph_{\omega+1}^{L(\bR)}$, $\delta^1_4=\aleph_{\omega+2}^{L(\bR)}$, for every $n$, $\delta^1_{2n+2}=((\delta^1_{2n+1})^+)^{L(\bR)}$, for every $n$, $\delta^1_{2n+1}$ is itself a successor cardinal of $L(\bR)$ and $\bS^1_{2n+2}$ is exactly the collection of $\d^1_{2n+1}$-Suslin sets\footnote{See \cite[Theorem 2.18, Chapter 3]{Jackson}.}. In \cite{delta15}, Jackson computed the remaining $\delta^1_n$'s and in particular, showed that $\delta^1_5=\aleph_{\omega^{\omega^\omega}+1}^{L(\bR)}$.
 
 In \cite{Kechris}, Kechris proved a partial generalization to Harrington's theorem\footnote{As pointed out by the referee, the theorems and conjectures stated below do not need ${\sf{DC}}$. See, for example, \cite{ADDC}. However, the original versions of these results assumed ${\sf{DC}}$, and so we also assume it.}.
\begin{theorem}[Kechris] Assume ${\sf{ZF+AD+DC}}$. Then $\delta^1_{2n+2}$ is $\mathbf{\Delta}^1_{2n+1}$-unreachable.
\end{theorem}

Moreover, in \cite{Kechris}, Kechris showed that there is no $f:\mathbf{\delta}^1_{2n+2}\rightarrow \bS^1_{2n+2}$ such that for all $\a, \b<\d^1_{2n+2}$, if $\a\not =\b$ then $f(\a)\cap f(\b)=\emptyset$. In Chapter 4 of \cite{Kechris}, Kechris conjectured that in fact a general form of Harrington-type theorem is true for projective pointclasses.

\begin{conjecture}[Kechris 1st Conjecture] Assume ${\sf{ZF+AD+DC}}$. Then $\bd^1_{2n+2}$ is $\bD^1_{2n+2}$-unreachable.
\end{conjecture}

In \cite{JacksonPart}, Jackson proved Kechris' 1st Conjecture by establishing the following remarkable theorem (see \cite[Corollary 4.5]{JacksonPart}).
\begin{theorem}[Jackson]\label{jackson} Assume ${\sf{ZF+AD+DC}}$. Then $\d^1_{2n+2}$ is $\bD^1_{2n+2}$-unreachable. 
\end{theorem}
Jackson's proof used his computation of the projective ordinals. In particular, that $\bd^1_{2n+1}$ has the strong partition property. In the last paragraph of \cite[Chapter 4]{Kechris}, Kechris made the following stronger conjecture.

\begin{conjecture}[Kechris 2nd Conjecture] Assume ${\sf{ZF+AD+DC}}$. Then $\bd^1_{2n+2}$ is $\bS^1_{2n+2}$-unreachable.
\end{conjecture}

This was partially resolved by Jackson and Martin who showed the following (see the Theorem on Page 84 of \cite{JacksonPart}).
\begin{theorem}[Jackson-Martin]\label{jacksonmartin} Assume ${\sf{ZF+AD+DC}}$. Then there is no strictly $\subset$-increasing or $\subset$-decreasing sequence $(A_\a:\a<\bd^1_{2n+2})\subseteq \bS^1_{2n+2}$. 
\end{theorem}

Chuang then generalized \rthm{jacksonmartin} by showing that in fact there are no strictly $\subset$-increasing or $\subset$-decreasing sequences of ${\mathbf{\bG}}$ sets of length $\d_\bG^+$ provided $\bG$ is closed under $\forall^\bR$, $\wedge$, $\vee$ and that $\bG$ has the pre-well-ordering property (see \cite[Theorem 3.5]{Jackson}). Here $\d_\bG$ is the supremum of the lengths of $\mathbf{\Delta}_{\bG}$-pwos\footnote{Here, $\mathbf{\Delta}_{\bG}$ consists of all those sets of reals $A$ such that both $A$ and $\bR-A$ belong to $\Gamma$.}. 

Kechris' 2nd Conjecture also appears in Kanamori's book where it appears as  \cite[Question 30.21]{Kanamori}. In \cite{Hjorth}, Hjorth verified Kechris' 2nd Conjecture for $n=0$ using techniques from inner model theory. In this paper, we  prove Kechris' 2nd Conjecture. 

\begin{theorem}[Hjorth-S.]\label{main theorem} Assume ${\sf{ZF+AD+DC}}$. Then $\bd^1_{2n+2}$ is $\bS^1_{2n+2}$-unreachable.
\end{theorem}

The following is an immediate corollary of \rthm{main theorem} and Woodin's result mentioned above\footnote{Notice that $\k$ is $\Gamma$-unreachable if and only if $\kappa$ is $\breve{\Gamma}$-unreachable.}. The case $n=0$ is due to Hjorth (\cite{Hjorth}).
\begin{corollary}\label{corollary}
Assume ${\sf{ZFC}}+{\sf{AD}}^{L(\bR)}$. If $E$ is a thin $\mathbf{\Pi}^1_{2n+2}$-equivalence relation then it has $\leq \bd^1_{2n+1}$-many equivalence classes.  In particular, any thin $\mathbf{\Pi}^1_2$ equivalence relation has $\leq \omega_1$-equivalence classes and any thin $\mathbf{\Pi}^1_4$ equivalence relation has $\leq \aleph_{\omega+1}^{L(\bR)}$-equivalence classes. 
 \end{corollary}

Our proof of \rthm{main theorem} uses inner model theory and directly builds on \cite{Hjorth} and \cite{SargPre}. 
From \cite{Hjorth}, we will mainly use the \textit{reflection argument}\footnote{Hjorth didn't call it a reflection argument.} used by Hjorth which appears on page 104 of \cite{Hjorth}. We state it as \rlem{hjorths trick}. According to Page 95 of \cite{Hjorth}, Hjorth's reflection argument is inspired by Woodin's unpublished proof of the pre-well-ordering property for $\Pi^1_3$. We strongly believe that it can have many other applications.

The reason that \rthm{main theorem} has been open since \cite{SargPre} is that the proof in \cite{Hjorth} uses the well-known Kechris-Martin theorem (see (i) and (ii) on page 105 of \cite{Hjorth}, see \cite{KM} for the Kechris-Martin theorem). It has been quite challenging to extend Kechris-Martin result in a way that could be useful to us. However, as it turns out, the use of Kechris-Martin theorem can be removed from \cite{Hjorth}, and this is our main new idea (see \rsubsec{remove km}). Clearly \rthm{jackson} is a corollary of \rthm{main theorem}, and so the inner model proof of \rthm{jackson} avoids the sophisticated machinery developed by Jackson in \cite{delta15}, though it uses inner model theory.

The main technical ingredient of our argument is the directed systems of mice. This is the system that Steel used in his calculation of $(\H|\Theta)^{L(\bR)}$ (\cite{Steel1995}) and Woodin used in his calculation of the $\H^{L(\bR)}$. The theory of these directed systems of mice has appeared in \cite{HODCoreModel}. We will use the material developed in \cite[Chapter 6]{HODCoreModel}. \cite{IMU} has a nice introduction to the subject. 

We expect that our methods will generalize and settle the following conjecture. ${\sf{AD}}^+$ is an extension of  ${\sf{AD}}$ introduced by Woodin (\cite{ADPlusBook}). 
\begin{conjecture}\label{my conjecture} Assume ${\sf{AD}}^+$. Suppose $\k$ is a regular Suslin cardinal and $\bG$ is the pointclass of $\k$-Suslin sets. Then $\k^+$ is $\bG$-unreachable. 
\end{conjecture}

\rcon{my conjecture} is a global conjecture like those made by Jackson (see \cite[Conjecture 6.4]{Jackson}, the conjectures in \cite{ICMJackson} and \cite[Problem 19]{OpenProblems}) though perhaps given the result of this paper \rcon{my conjecture} is somewhat easier than those made by Jackson. Such global conjectures test our understanding of projective sets. It is one of the deepest mysteries of descriptive set theory that, assuming ${\sf{AD}}$, the complete theory of analytic and  co-analytic sets doesn't immediately generalize to projective hierarchy. Perhaps the most well-known example of this phenomenon is that $\Delta^1_3$ to $(\Pi^1_3, \Sigma^1_3)$ is not the same as $\Delta^1_1$ to $(\Pi^1_1, \Sigma^1_1)$ (see \cite{Qtheory}). From a current point of view, it seems that the function $x\mapsto x^\#$ provides  singularly magical coding of subsets of $\omega_1$, and that coding,- which was used by Martin to establish the strong partition property for $\omega_1$ (see \cite[Theorem 28.12]{Kanamori}) and by Kechris-Martin to establish their celebrated Kechris-Martin theorem (see  \cite{KM}),- doesn't yet have a proper inner model theoretic generalization to higher levels of the projective hierarchy and beyond. Our current understanding is based on Jackson's deep analysis of measures (see \cite{Jackson}). The fact that global conjectures such as \rcon{my conjecture} and  Jackson's conjectures are still open seems to suggest that our current understanding of the projective hierarchy, just like it was  with our understanding of $(\Pi^1_1, \Sigma^1_1)$, may not be the final one, as whatever methods we discover to settle these global conjectures, the projective case will have to be the special case of these conjectures, and so these yet-to-be-discovered ideas will come with new insight into the projective hierarchy. \\\\

\textbf{Acknowledgments.} I wish to thank John Steel for introducing me to \cite{BoundHjorth} so many years ago. The work carried out in \cite{SargPre}, which answers most of the questions raised in the addendum of \cite{BoundHjorth}, was done while I was Steel's PhD student. The addendum of \cite{BoundHjorth} appeared in the unpublished version of Hjorth's paper by the same title available on his web site. I am grateful to Derek Levinson for a list of typos and corrections, and indebted to the referee for a long list of corrections.

This paper is about 10 years late. Sometime in 2010, while I was a postdoc at UCLA, Hjorth suggested that two of us work on improving the results of \cite{SargPre}, and in particular compute $b_{2n+1, 0}$ (which \cite{SargPre}  conjectures to be $\d^1_{2n+2}$) and also prove Kechris' 2nd Conjecture.  In \cite{BoundHjorth}, Hjorth showed that $b_{1, 0}=\delta^1_2=\omega_2$. The issues we encountered were the familiar ones: proving Kechris-Martin for $\Pi^1_5$ and beyond, and avoiding fine boundedness arguments involving $\Sigma^1_1$ relations and admissible ordinals. Unfortunately, on January 13 of 2011, Greg Hjorth unexpectedly passed away, and the project has remained unfinished. In the Spring of 2021 it became apparent that the use of Kechris-Martin in \cite{BoundHjorth} is unnecessary. 

The author’s work is funded by the National Science Center, Poland under the Maestro Call, registration number UMO-2023/50/A/ST1/00258.

\section{Remarks, Notations and Terminology}\label{remarks}

Experts can skip this and the following sections and directly go to \rsec{the proof}. The results in \rsec{pi iterability} are not fundamentally new and go back to \cite{PWIM}. However, \cite{PWIM} doesn't state them in the exact form that we need. Below we make some remarks, and set up our notation and terminology. 

\begin{review}\label{basic imt}\textbf{Basic concepts from inner model theory:}
\end{review}
The reader unfamiliar with basic concepts of inner model theory might find it helpful to consult \cite{OIMT}. Also, the introduction of \cite{IMU} is accessible and introduces many of the concepts that we need. Extenders were treated both in \cite{Jech} and \cite{Kanamori}.
\begin{enumerate}
\item Suppose $x\in {\sf{HC}}$\footnote{$\sf{HC}$ is the set of hereditarily countable sets.} is such that there is a wellordering of $x$ in $L_1[x]$ which can be defined without parameters over $L_0[x]$\footnote{Let $L_0[x]$ be the the transitive closure of $x$.}. We say $(\M, \Phi)$ is an $x$-mouse pair if $\M$ is an $x$-premouse and $\Phi$ is an $(\omega_1, \omega_1)$-iteration strategy for $\M$ (for example, see \cite[Definition 2.19, 3.9 and 4.4]{OIMT}\footnote{Here we only consider $(\omega, \omega_1, \omega_1)$-iteration strategies, and so we drop $``\omega"$ for convenience.})\footnote{An $x$-premouse is defined similarly to a premouse except one requires that $\mathcal{J}_0^{\M}$ is the transitive closure of $x$}. We say $(\M, \Phi)$ is a countable mouse pair if $\M$ is countable. We will often say that $\M$ is a premouse or $(\M, \Phi)$ is a mouse pair without mentioning the $x$. 
\item Suppose $\M=\mathcal{J}_\b^{\vec{E}}$ is a premouse and $\a\leq {\sf{Ord}}\cap \M$. We let $\vec{E}^\M$ be the extender sequence of $\M$. Because we will allow padded iterations, we let $\dom(\vec{E})^\M=\b$ and for those $\gg$ such that $\M$ doesn't have an extender indexed at $\gg$, we set $\vec{E}^\M(\gg)=\emptyset$. We then let $\M|\a=(\mathcal{J}_{\omega\a}^{\vec{E}\rest \omega \a}, \vec{E}\rest \omega\a, \in)$ and $\M||\a=(\mathcal{J}_{\omega \a}^{\vec{E}\rest \omega \a}, \vec{E}\rest \omega\a, \vec{E}(\omega \a), \in)$. 
\item Suppose $\M$ is a premouse. We say $\eta$ is a cutpoint of $\M$ if for all $E\in \vec{E}^\M$ with the property that $\cp(E)<\eta$, $\lh(E)\leq \eta$. 
\item Under ${\sf{AD}}$, as $\omega_1$ is measurable, $\omega_1$-iterability\footnote{This concept is defined in \cite[Definition 3.9]{OIMT}. Here again we ignore the integer $k$ of that definition as we take it to be $\omega$. If $\k$ is a measurable cardinal then $\k$-iterability implies $\kappa+1$-iterability.} is what is needed to prove the ${\sf{Comparison\ Theorem}}$ for countable mouse pairs (see \cite[Theorem 3.11]{OIMT}).
\item For a definition of an iteration\footnote{Following Jensen, we will use iteration for iteration trees.} and an iteration strategy see \cite[Definition 3.3 and 3.9]{OIMT}. Given an iteration $\T$ of a premouse $\M$, we write $\T=((\M_\a: \a<\lh(\T)), (E_\a: \a+1<\lh(\T)), \mathcal{D}, T)$  where
\begin{enumerate}
\item $E_\a\in \vec{E}^{\M_\a}$ is the extender picked from $\M_\a$,
\item $\mathcal{D}$ is the set of those $\a$ where a drop occurs, and
\item $T$ is the tree order of $\T$.
\end{enumerate}
We allow padded iterations, and so it is possible that $E_\a=\emptyset$.
\item Suppose $(\M, \Phi)$ is a mouse pair and $\N$ is a $\Phi$-iterate of $\M$ via iteration $\T$. We then let $\Phi_{\N, \T}$ be the strategy of $\N$ induced by the pair $(\Phi, \T)$. More precisely, $\Phi_{\N, \T}(\U)=\Phi(\T^\frown \U)$. 
\item Continuing with $(\M, \Phi)$, $\N$ and $\T$ as above, we say $\N$ is a complete $\Phi$-iterate if the iteration embedding $\pi^\T$ is defined, which happens if and only if there is no drop on the main branch of $\T$ (see the paragraph after \cite[Definition 3.3]{OIMT}).
\item We say $\M$ is an almost knowledgable mouse if $\M$ has a unique $(\omega_1, \omega_1)$-iteration strategy $\Phi$ such that whenever $\T$ is an iteration of $\M$ via $\Phi$ and $\N$ is the last model of $\T$ then 
\begin{enumerate}
\item  $\Phi_{\N, \T}$ is independent of $\T$, and\footnote{Such strategies are usually called positional, see for example \cite[Chapter 2.6]{ATHM}.}
\item if $\N$ is a complete iterate of $\M$ then $\pi^\T$ is independent of $\T$\footnote{Such strategies are usually called commuting, see for example \cite[Chapter 2.6]{ATHM}.}. 
\end{enumerate}
\item We say $\M$ is knowledgable if letting $\Phi$ be the unique $(\omega_1, \omega_1)$-iteration strategy of $\M$, whenever $\N$ is a complete iterate of $\M$ via $\Phi$, $\N$ is almost knowledgable.
\item If $\M$ is knowledgable then we let $\Phi_\M$ be its unique $(\omega_1, \omega_1)$-iteration strategy and for each $\Phi$-iterate $\N$ of $\M$, we let $\Phi_\N=\Phi_{\N, \T}$ where $\T$ is some iteration of $\M$ via $\Phi$ with last model $\N$. If $\M$ is knowledgable and $\N$ is a $\Phi_\M$-iterate of $\M$ then we say that $\N$ is an iterate of $\M$. 
If $\N$ is a complete $\Phi_\M$-iterate of $\M$ then we  say that $\N$ is a complete iterate of $\M$. 
\item Suppose $\M$ is a  knowledgable mouse and $\N$ is a normal iterate\footnote{We say $\N$ is a \textit{normal iterate} of $\M$ if there is a normal iteration tree $\T$ on $\M$ whose last model is $\N$. Normal iteration trees are exactly the iteration trees defined in \cite{OIMT} on page 23. In such trees, the lengths of the extenders increase and each extender is applied to the earliest possible model appearing in the tree. Normal iteration trees are the trees that appear in the comparison process.} of $\M$. Then we let $\T_{\M, \N}$ be the unique normal iteration of $\M$ according to $\M$'s unique iteration strategy whose last model is $\N$\footnote{$\T_{\M, \N}$ is unique because it is the iteration of $\M$ that is build via the comparison process, see \cite[Chapter 3.2]{OIMT}.}. In general, if $\N$ is an iterate of $\M$ then $\M$-to-$\N$ iteration may not be unique\footnote{\cite{farm2} establishes the following remarkable theorem. Suppose that $(\M, \Phi)$ is a mouse pair, $\M$ is sound and projects to $\omega$ and $\Phi$ is the unique iteration strategy of $\M$. Then $\Phi$ has full normalization, i.e., every $\Phi$-iterate of $\M$ can be obtained as a normal $\Phi$-iterate of $\M$.}. If $\N$ is a complete iterate of $\M$, we let $\pi_{\M, \N}:\M\rightarrow \N$ be the iteration embedding.
\item Suppose $\M$ is an $x$-premouse and $\T$ is an iteration of $\M$ of limit length. We let $\d(\T)=\sup\{\lh(E_\a^\T): \a<\lh(\T)\}$ and $\c(\T)=\cup_{\a<\lh(\T)}\M_\a^\T|\lh(E_\a^\T)$\footnote{$\c(\T)$ is usually denoted by $\M(\T)$. However, $\M$ gets overused in inner model theory, so we decided to change the notation.}. For more on these objects see \cite[Definition 6.9]{OIMT}.
\item Suppose $x\in \bR$. $\M_n(x)$ is the minimal class size $x$-mouse with $n$ Woodin cardinals. $\M_n^\#(x)$ is the minimal active\footnote{This just means that it has a last extender predicate indexed at the ordinal height of the mouse.} $x$-mouse with $n$ Woodin cardinals. We say $\M_n^\#(x)$ exists if there is an $\omega_1+1$-iterable active $x$-premouse with $n$ Woodin cardinals. Assuming ${\sf{AD}}$, $\M_n^\#(x)$ exists (for example, see \cite{PDMice} or \cite[Sublemma 3.2]{CoarseAD}). For each $n$, $\M_n^\#(x)$ is knowledgable (assuming it exist). If every real has a sharp and $\M_n(x)$ exists then $\M_n(x)$ is knowledgable. For the proof of these and relevant results see \cite[Chapter 2 and 3]{HODCoreModel} and especially \cite[Theorem 3.23]{HODCoreModel}.
\item Suppose $\M$ is an $x$-premouse, $\T$ is an iteration of $\M$ of limit length and $b$ is a branch. We say $\Q(b, \T)$ exists if there is $\a$ such that $\M^\T_b||\a\models ``\d(\T)$ is a Woodin cardinal" but $\mathcal{J}_1[\M^\T_b||\a]\models ``\d(\T)$ is not a Woodin cardinal". If $\Q(b, \T)$ exists we let it be $\M^\T_b||\a$ where $\a$ is the largest such that $\M^\T_b||\a\models ``\d(\T)$ is a Woodin cardinal". 
\item $\Q(b, \T)$ defined above is one of the most used objects in inner model theory. The reader may want to consult \cite[Definition 6.11]{OIMT} and \cite[Definition 2.11]{PWIM}. The reason that $\Q(b, \T)$ is important is that it uniquely identifies $b$. More precisely, if $c\not=b$ is another cofinal branch of $\T$ such that $\Q(c, \T)$ exists then $\Q(c, \T)\not =\Q(b, \T)$.
\item Suppose $\M$ is an $x$-premouse and $\T$ is an iteration of $\M$. We say $\T$ is below $\d$ if for every $\a<\lh(\T)$ either $\pi_{0, \a}^\T$ is not defined or $E_\a^\T\in \vec{E}^{\M_\a^\T|\pi_{0, \a}^\T(\d)}$. We say $\T$ is above $\nu$ if for every $\a<\lh(\T)$, $\cp(E_\a^\T)\geq \nu$.
\item We remark that when we say that $``\k$ is a measurable cardinal in a premouse $\M"$ or $``\k$ is a strong cardinal in a premouse $\M"$ or say other similar expressions we tacitly assume that these large cardinal properties are witnessed by the extenders on the extender sequence of $\M$. See \cite{MeasuresInMice} for results showing that such a restriction is unnecessary. 
\end{enumerate}
\begin{review}\label{ds} \textbf{The directed system:}
\end{review}
 The theory of directed systems, by now, has a long history. It originates in Steel's seminal paper \cite{Steel1995}. Since then it has been used to establish a number of striking applications of inner model theory to descriptive set theory. To learn more about the subject the interested reader may consult \cite[Chapter 8]{OIMT}, \cite[Chapter 6]{HODCoreModel}, \cite{Sandra}, \cite{IMU}, \cite{BSL}, \cite{ATHM}, \cite{Varsovian} and many other sources.
\begin{enumerate}
\item Suppose $\P$ is a knowledgable mouse. Let $\mathcal{I}_\P$ be the set of complete iterates $\N$ of $\P$ such that $\P$-to-$\N$ iteration has a countable length.
\item Suppose that either there is some $\nu<\omega_1$ such that $\P=L[\P|\nu]$ or  $\P$ itself is countable.  Then comparison implies that if $\R, \S\in \mathcal{I}_\P$ then there is  $\W\in \mathcal{I}_\P$ such that $\W$ is a complete iterate of both $\R$ and $\S$. Define $\leq_\P$ on $\mathcal{I}_\P$ by setting $\R\leq_\P \S$ if and only if $\S$ is an iterate of $\R$. Assuming $\P$ is knowledgable, we then get a directed system $\mathcal{F}_\P$ whose models consist of the models in $\mathcal{I}_\P$, whose directed order is $\leq_\P$ and whose embeddings are the iteration embeddings $\pi_{\R, \S}$. 
\item  Assuming $\P$ is as above, we let $\M_\infty(\P)$ be the direct limit of $\mathcal{F}_\P$ and given $\N\in \mathcal{I}_\P$, we let $\pi_{\N, \infty}:\N\rightarrow \M_\infty(\P)$ be the iteration embedding according to $\Phi_\N$. ${\sf{Comparison\ Theorem}}$ (see \cite[Theorem 3.11]{OIMT}) implies that $\M_\infty(\P)$ is well-founded (see for example the remark after \cite[Definition 8.15]{OIMT})).
\end{enumerate}
\begin{review}\label{ea} \textbf{The extender algebra:}
\end{review}
The extender algebra, which was discovered by Woodin, is the magic tool of inner model theory. The reader may consult \cite[Chapter 7.2]{OIMT}.
\begin{enumerate}
\item Suppose $\P$ is a premouse, $\d$ is a Woodin cardinal of $\P$ and $\nu<\d$. 
\begin{definition}\label{wap}
We say that $\mathcal{E}$ is \textbf{weakly appropriate} at $\d$ if $\mathcal{E}$ is a set consisting of extenders $E\in \vec{E}^{\P|\d}$ such that\\\\
(a) $\nu(E)$\footnote{Recall from \cite[Definition 2.2]{OIMT} that $\nu(E)$ is the supremum of the generators of $E$, the natural length of $E$.} is an inaccessible cardinal in $\P$,\\
(b) $\mathcal{E}$ witnesses that $\d$ is a Woodin cardinal\footnote{In the sense that for every $A\subseteq \d$ there is an $E\in \mathcal{E}$ such that $A\cap \nu(E)=\pi_E(A)\cap \nu(E)$.}.\\\\
If in addition\\\\
(c) for each $E\in \mathcal{E}$, $\pi_E(\mathcal{E})\cap (\P|\nu(E))=\mathcal{E}\cap (\P|\nu(E))$,\\\\
then we say that $\mathcal{E}$ is \textbf{appropriate}.
\end{definition}
 When $\d$ is clear from the context we will omit the expression ``at $\d$". Suppose now that $\mathcal{E}$ is weakly appropriate. We then let ${\sf{Ea}}^\P_{\d, \nu, \mathcal{E}}$ be the extender algebra of $\P$ defined using extenders $E\in \mathcal{E}$ such that $\cp(E)>\nu$. If $\nu=0$ then we omit it from our notation. If $\mathcal{E}$ consists of all extenders or if its role is irrelevant then we omit it from the notation.
 
 ${\sf{Ea}}^\P_{\d, \nu, \mathcal{E}}$ is the basic extender algebra for adding a real: thus, it has only countably many predicate symbols. The conditions in ${\sf{Ea}}^\P_{\d, \nu, \mathcal{E}}$ are formulas, and so given $g\subseteq {\sf{Ea}}^\P_{\d, \nu, \mathcal{E}}$ and $r\in {\sf{Ea}}^\P_{\d, \nu, \mathcal{E}}$, we will often write $g\models r$ instead of $r\in g$. 
 
 Instead of making our notation endlessly complicated, we will abuse our terminology and notation in the following way. Suppose $a=L[b]$ where for some $\a$, $b\subseteq \a$. Then we will say that $a$ is generic for ${\sf{Ea}}^\P_{\d, \nu, \mathcal{E}}$ to mean that if $\a$ is the least such that there is $b\in a$ with the property that $b\subseteq \a$ and $a=L[b]$, then $b$ is generic for the extender algebra at $\d$ that uses $\a$ many generators and extenders who critical points are strictly greater than $\nu$.
 
 A celebrated theorem of Woodin says that  ${\sf{Ea}}^\P_{\d, \nu, \mathcal{E}}$  has the $\d$-c.c. condition (assuming only that $\mathcal{E}$ is weakly appropriate, see  \cite[Chapter 7.2]{OIMT}).
\item Suppose $(\P, \Sigma)$ is a mouse pair, $\d$ is a Woodin cardinal of $\P$, $\nu<\d$, $\mathcal{E}$ is an appropriate\footnote{Here, we need appropriateness to ensure that the resulting iteration is normal.} set of extenders  and $(x_1, ..., x_k)\in \bR^k$. We say $\T$ is the $(\vec{x}, \d, \nu, \mathcal{E})$-genericity iteration of $\P$ if $\T$ is according to $\Sigma$ and for each $\a<lh(\T)$, setting $\M_\a=_{def}\M_\a^\T$ and $E_\a=_{def}E_\a^{\T}\in \vec{E}^{\M_\a}$, $\lh(E_\a)$ is the least $\gg\in \dom(\vec{E}^{\M_\a})$ such that setting $E=\vec{E}^{\M_\a}(\gg)$, the following clauses hold:
\begin{enumerate}
\item $\cp(E)>\nu$ and $\lh(E)<\pi^\T_{0, \a}(\d)$,
\item $E\in \pi_{0, \a}^\T(\mathcal{E})$ (and hence, $E$ measures all subsets of $\cp(E)$ in $\M_\a$),
\item for some $i\leq k$, $x_i$ doesn't satisfy an axiom of $\pi_{\P, \M_\a}({\sf{Ea}}^\P_{\d, \nu. \mathcal{E}})$ that is generated by $E$. More precisely,  $x_i\not \models A_{E, \vec{\phi}}$ where $\vec{\phi}\in\M_\a|(\cp(E)^+)^{\M_\a}$ and $A_{E, \vec{\phi}}$ is the axiom $\bigvee \vec{\phi}\iff \bigvee \pi_{E}^{\M_\a}(\vec{\phi})\rest \nu(E)$, and
\item for all $\gg'<\gg$, if $\gg'\in \dom(\vec{E}^{\M_\a})$ then $E_{\gg'}^{\M_\a}$ does not satisfy clauses (a)-(c) above.
\end{enumerate}
\item  Suppose $\phi(v_0, ..., v_k)$ is a $\Sigma^1_{2n+2}$ formula, $\M$ is a premouse,  $\d_0<...<\d_{2n-1}$ are Woodin cardinals of $\M$, $\k<\d_0$, $d=(\k, \d_0, ..., \d_{2n-1})$ and $\vec{a}\in [\bR^{\M}]^{<\omega}$. By induction, we define $\phi_{\M, d}$ and the meaning of $\M\models \phi_{\M, d}[\vec{a}]$.  If $n=0$ then $\phi_{\M, d}=\phi$ and $\M\models \phi_{\M, d}[\vec{a}]$ if and only if $\M\models \phi[\vec{a}]$. Next let $\psi$ be a $\Sigma^1_{2n}$ formula such that $\phi(v_0, ..., v_k)\iff \exists u_0 \forall u_1 \psi(u_0, u_1, v_0, ..., v_k)$. We then write $\M\models \phi_{\M, d}[\vec{a}]$ if and only if there is  $p\in {\sf{Ea}}_{\d_0, \k}$ such that if $g$ is the name of the generic for ${\sf{Ea}}_{\d_0, \k}$ then $p$ forces that there is $x\in \bR$ such that every $q\in {\sf{Ea}}^{\M[g]}_{\d_1, \d_0}$ forces that for all $y\in \bR$, $\psi_{\M, d'}[ x, y, \vec{a}]$ where $d'=(\d_1, \d_2,..., \d_{2n-1})$.\footnote{We believe that the reader can extract the exact formula $\phi_{\M, d}$ from this equivalence.} Similarly we can define $\phi_{\M, d}$ for a $\Pi^1_{2n+3}$ formula $\phi$ and $d=(\k, \d_0, \d_1, ..., \d_{2n})$ (assuming that $\d_0, \d_1, ..., \d_{2n}$ are Woodin cardinals of $\M$). 

\item We will need the following basic applications of the extender algebra. 

\begin{proposition}\label{correctness} Suppose $x\in \bR$, $\M$ is countable $\omega_1+1$-iterable mouse over $x$, $\M\models {\sf{ZFC}}$, $\M$ has $2n$ Woodin cardinals, and  $\phi(\vec{v})$ is a $\Sigma^1_{2n+2}$ formula. Let $\d_0<\d_1<...<\d_{2n-1}$ be the Woodin cardinals of $\M$ and let $\k<\d_0$ be any ordinal. Set $d=(\k, \d_0, ..., \d_{2n-1})$ and suppose that for some $\vec{a}\in \M\cap \bR^k$, $\M\models \phi_{\M, d}[\vec{a}]$. Then $\phi[\vec{a}]$. 
\end{proposition}
\begin{proof} We give the proof of the prototypical case of $n=1$. Suppose $\phi$ is $\exists u_0 \forall u_1 \psi[u_0, u_1, \vec{v}]$ where $\psi$ is $\Sigma^1_2$.  Let $p\in {\sf{Ea}}^\M_{\d_0, \k}$ be a condition witnessing $\phi_{\M, d}[\vec{a}]$. Let  $g\subseteq {\sf{Ea}}^\M_{\d_0, \k}$ be $\M$-generic such that $p\in g$ and $g\in V$. Let then $b_0$ be a real in $\M[g]$ such that every $q\in {\sf{Ea}}^{\M[g]}_{\d_1, \d_0}$ forces that if $u$ is a real then $\psi[b_0, u, \vec{a}]$. 

We now want to see that for all $b_1\in \bR$, $\psi[b_0, b_1, \vec{a}]$ holds. Fix $b_1\in \bR$ and an $\omega_1+1$-strategy $\Sigma$ for $\M$. Let $\N$ be a complete iterate of $\M$ via iteration $\T$ that is above $\d_0$ and is such that $b_1$ is generic over $\N$ for $\pi^\T({\sf{Ea}}^{\M[g]}_{\d_1, \d_0})$. Because $\cp(\pi^\T)>\d_0$, we have that in $\N[g]$,  every $q\in\pi^\T({\sf{Ea}}^{\M[g]}_{\d_1, \d_0})$ forces that if $u$ is a real then $\psi[b_0, u, \vec{a}]$. Therefore, we have that $\N[g][b_1]\models \psi[b_0, b_1, \vec{a}]$. But then it follows from the upward absoluteness of $\Sigma^1_2$ formulas that $\psi[b_0, b_1, \vec{a}]$. 
\end{proof}
The same proof also gives the following. 
\begin{proposition}\label{correctness pi} Suppose $x\in \bR$, $\M$ is countable $\omega_1+1$-iterable mouse over $x$, $\M\models {\sf{ZFC}}$, $\M$ has $2n+1$ Woodin cardinals, and  $\phi(\vec{v})$ is a $\Pi^1_{2n+3}$ formula. Let $\d_0<\d_1<...<\d_{2n}$ be the Woodin cardinals of $\M$ and let $\k<\d_0$ be any ordinal. Set $d=(\k, \d_0, ..., \d_{2n})$ and suppose that for some $\vec{a}\in \M\cap \bR^k$, $\M\models \phi_{\M, d}[\vec{a}]$. Then $\phi[\vec{a}]$. 
\end{proposition}
\item In \cite{Hjorth}, Hjorth showed that the product of the extender algebra with itself is still $\d$-c.c. More precisely, suppose $M$ is a transitive model of ${\sf{ZFC}}$, $\d$ is a Woodin cardinal of $M$, $\mathcal{E}$ is a weakly appropriate set of extenders and $\nu, \nu'<\d$. Then ${\sf{Ea}}^M_{\d, \nu, \mathcal{E}}\times {\sf{Ea}}^M_{\d, \nu', \mathcal{E}}$ is $\d$-c.c. For the proof see \cite[Lemma 1.2]{Hjorth}. 
\item Suppose $M$, $\d$ and $\nu, \nu'<\d$ are as above. We then let ${\sf{ea}}$ be the name for the generic object for ${\sf{Ea}}^M_{\d, \nu}$ and let $({\sf{ea}}_1, {\sf{ea}}_2,..., {\sf{ea}}_n)$ be the sequence of reals ${\sf{ea}}$ codes. We let $({\sf{ea}}^l, {\sf{ea}}^r)$ be the name for the generic of ${\sf{Ea}}^M_{\d, \nu}\times {\sf{Ea}}^M_{\d, \nu'}$. We then have that $({\sf{ea}}^l_1,..., {\sf{ea}}^l_n)$ and $(\sf{ea}^r_1,..., {\sf{ea}}^r_m)$ are the finite sequences of reals coded by ${\sf{ea}}^l$ and ${\sf{ea}}^r$. 
\end{enumerate}

\begin{review}\label{s constructions} \textbf{Backgrounded constructions and $S$-constructions:}
\end{review}
$S$-construction is a method of translating the mouse structure to a similar structure over some inner model. The details of such a construction first appeared in \cite[Lemma 1.5]{Selfiterability} where it was called $P$-constructions. In \cite{SargPre}, the author renamed them $S$-constructions\footnote{``S" stands for Steel.}. In this paper, we will need a special instance of $S$-constructions. The reader is advised to review the notion of \textit{fully backgrounded constructions}\footnote{The difference between full backgrounded constructions and $K^c$ constructions defined in \cite{OIMT} as well as in many other places is that in the fully backgrounded constructions the background extenders are assumed to be total (i.e. they measure all subsets of their critical points) while in $K^c$ constructions this requirement is relaxed. Thus, the fully backgrounded constructions can inherit large cardinals only if there are already large cardinals in $V$ while $K^c$ constructions can inherit large cardinals even if there are no large cardinals in $V$. For example, see \cite{JSSS}.} as defined in \cite[Chapter 11]{FSIT}.
\begin{enumerate}
\item Suppose $\P$ is a premouse, $\d$ is a $\P$-cardinal and $z\in \P|\d$. We then let ${\sf{Le}}(z)$ be the last model\footnote{We will often say that it is the output of the fully backgrounded construction.} of the fully backgrounded construction of $\P|\d$ over $z$ as defined in \cite[Chapter 11]{FSIT}. In this construction, all extenders used have critical points $>\eta$ where $\eta$ is the least such that $z\in \P|\eta$. If $z=\emptyset$ then we omit it from our notation. We will only consider such fully backgrounded constructions over $z$ that can be easily coded as a subset of ordinals. Typical examples of $z$'s that we will consider are reals and premice.

 If $\P$ is $\omega_1+1$-iterable then ${\sf{Ord}}\cap {\sf{Le}}(z)=\d$, and if $\a={\sf{Ord}}\cap \P$ and $\d$ is Woodin in $\P$ then $L_\a[{\sf{Le}}(z)]\models ``\d$ is a Woodin cardinal" (see \cite[Chapter 11]{FSIT}). 

Below, in clauses 2-6, $\P$, $\d$ and $z$ are as above, $\d$ is a Woodin cardinal of $\P$ and $u\in \bR\cap \P$.
\item  Suppose the $z$-premouse $\M\subseteq \P$ is such that ${\sf{Le}}(z)\insegeq \M$ and for every $\a<{\sf{Ord}}\cap \M$, $\M||\a \in \P$. Then for any $\nu<\d$, $u$ is generic over ${\sf{Ea}}^\M_{\d, \nu}$. This is because every extender $E\in \vec{E}^{{\sf{Le}}(z)}$ such that $\nu(E)$ is inaccessible in ${\sf{Le}}(z)$ is background by an extender $F\in \vec{E}^\P$ such that there is a factor map $\tau: Ult({\sf{Le}}(z), E)\rightarrow \pi^\P_F({\sf{Le}}(z))$ with the property that $\cp(\tau)\geq \nu(E)$. Fixing now $\vec{\phi}\in {\sf{Le}}(z)|(\cp(E)^+)^{{\sf{Le}}(z)}$ such that $u\models \bigvee \pi_E^{{\sf{Le}}(z)}(\vec{\phi})\rest \nu(E)$, we have that $u\models \bigvee \pi^\P_F(\vec{\phi})$. But now because $u\in \P$, we have that $u\models \bigvee \vec{\phi}$. Thus, $u\models A_{E, \vee{\phi}}$.
\item We say $\P$ is translatable if for every $z\in \P\cap \bR$, ${\sf{Ord}}\cap {\sf{Le}}(z)=\d$\footnote{The fully backgrounded constructions may fail to reach $\d$.}. 
\item Given a translatable $\P$ and $z\in \P\cap \bR$, we let ${\sf{StrLe}}(\P, z)$ be the result of the $S$-construction over ${\sf{Le}}(z)$ that translates the extenders of $\P$ with critical points $>\d$ into extenders over ${\sf{Le}}(z)$. 
\item It is shown in \cite[Lemma 1.5]{Selfiterability} that if $E\in \vec{E}^\P$ is such that $\cp(E)>\d$ then $E\cap {\sf{StrLe}}(\P, z)\in \vec{E}^{{\sf{StrLe}}(\P, z)}$.
\item It follows that every Woodin cardinal of $\P$ greater than $\d$ is a Woodin cardinal of ${\sf{StrLe}}(\P, z)$, and also if $\P$ is $\omega_1$-iterable then ${\sf{StrLe}}(\P, z)$ is $\omega_1$-iterable.
\item Suppose $x, z\in \bR$, $\N$ is a complete iterate of $\M_n^\#(z)$ such that $\T_{\M_n^\#(z), \N}$ is below the least Woodin cardinal of $\M_n^\#(z)$, and some real recursive in $x$ codes $\N$. Set $\P={\sf{StrLe}}(\M_n(x), z)$. Then $\P$ is a complete iterate of $\N$. The proof proceeds as follows. First it is shown that there is a normal iteration $\T$ of $\N$ which is below the least Woodin cardinal of $\N$ and if $\P'$ is the last model of $\T$ then $(\N(z))^{\M_n(x)}=\P'|\d$ where $\d$ is the least Woodin cardinal of $\P'$ (and also $\M_n(x)$). To establish this result one uses the stationarity of the backgrounded constructions which says that in the comparison of $\P$ and $\P'$, the iteration of $\P$ is trivial. One then shows that $\P'=\P$, and here, the important fact is that $\M_n(x)|\d$ is generic over $\P$. This in particular implies that $\P[\M_n(x)|\d]=\M_n(x)$. One then concludes that every set in $\P$ is definable from a finite sequence $s\in \d^{<\omega}$ and a finite sequence of indiscernibles for $\M_n(x)$. Since $\P'|\d=\P|\d$ and $\P'$ has the same property (being a complete iterate of $\N$ below its least Woodin cardinal), it follows that $\P=\P'$. The details of what we have said has appeared in a number of places. The reader may find it useful to consult \cite[Lemma 3.20]{Sandra}, \cite[Definition 1.1]{NegIdeal}, \cite{Varsovian}, \cite[Lemma 2.11]{ATHM}, \cite[Lemma 3.23]{trang2013}, \cite[Lemma 1.3]{Selfiterability} and the discussion after \cite[Lemma 1.4]{Selfiterability}.\\

The following objects will be used in clauses 8-10. Suppose $\P$ is a premouse, $\d$ is a Woodin cardinal of $\P$ and $a\in \P|\d$. Let $\N=({\sf{Le}}(a))^{\P|\d}$ and suppose ${\sf{Ord}}\cap \N=\d$.\footnote{In this construction all extenders used have critical points greater than $\eta$ where $\eta$ is least such that $a\in \P|\eta$.}  
\item Suppose $\k$ is a measurable cardinal of $\N$ as witnessed by the extenders on the sequence of $\N$. Then $\k$ is a measurable cardinal in $\P$. This is essentially because in the fully backgrounded construction all extenders used for backgrounding purposes are total extenders. 
\item Similarly, if $\k$ is a strong cardinal of $\N$ then $\kappa$ is a $<\d$-strong cardinal in $\P$.
\item Suppose $\k$ is a $<\d$-strong strong cardinal in $\P$. Then
\begin{lemma}\label{strongs are fb} $\N|\k=({\sf{Le}}(a))^{\P|\k}$. 
\end{lemma}
\begin{proof} To see this, suppose not. Set $\N'=({\sf{Le}}(a))^{\P|\k}$ and let $\xi$ be the least such that the $\xi$th model of the fully backgrounded construction of $\P|\d$ over $a$ projects across $\k$. Let $\Q$ be the $\xi$th model of the fully backgrounded construction of $\P$ over $a$ and let $\nu<\d$ be such that $\Q$ is constructed by the fully backgrounded construction of $\P|\nu$ over $a$. Let $\Q'$ be the core of $\Q$. Since $\rho_\omega(\Q)<\k$, we must have that $\Q'\in \P|\k$. We now have that $\Q'$ is not constructed by the fully backgrounded construction of $\P|\k$ over $a$\footnote{Because the construction reaches $\d$ we cannot construct the same model twice at different stages as otherwise the construction will be looping between these two stages.}. However, if $E\in \vec{E}^{\P|\d}$ is any extender with $lh(E)>\nu$, then in $Ult(\P, E)$, $\pi_E(\Q')=\Q'$ is constructed by the fully backgrounded construction of $Ult(\P, E)|\pi_E(\k)$ over $a$.
\end{proof}

\begin{definition}\label{fbcut} We say $\k$ is an $fb$-cut in $\P$ if letting $\d_0$ be the least Woodin cardinal of $\P$, $\k<\d_0$, $\k$ is a $\P$-cardinal and
\begin{center}
${\sf{Le}}^{\P|\d_0}|\k={\sf{Le}}^{\P|\k}$.
\end{center} 

We say $\k$ is a weak $fb$-cut if  whenever $\Q$ is a mouse that appears in the fully backgrounded construction of $\P$ over ${\sf{Le}}^{\P|\k}$ (and hence uses extenders with critical points $>\k$), $\rho_\omega(\Q)\geq \k$\footnote{$\k$ must be a cutpoint in such a $\Q$.}. 
\end{definition}

\item (\textbf{Universality})  Suppose $\P$ is a translatable premouse (see clause 3 above), $\d$ is its least Woodin cardinal and $a\in \P|\d$. Suppose $\Q$ is a $\d+1$-iterable $a$-premouse in $\P$. Then either ${\sf{Le}}(a)$ has a superstrong cardinal or $\Q\insegeq {\sf{Le}}(a)$. Moreover, if $\N$ is some fully backgrounded construction of $\P|\d$ such that $a\in \N$ and ${\sf{Ord}}\cap {\sf{Le}}(a)^\N=\d$ then either ${\sf{Le}}(a)^\N$ has a superstrong cardinal or $\Q\insegeq {\sf{Le}}(a)^\N$. These results are due to Steel and are consequences of universality of the fully backgrounded constructions (e.g. see \cite[Lemma 2.12 and 2.13]{ATHM}).
\item Let $\P=\M_n(x)$ where $n>0$ and $x\in \bR$, and let $\d$ be the least Woodin cardinal of $\P$. Let $\eta<\d$. Then $\P|\eta$ is $\d+1$-iterable inside $\P$. This is because if $\T$ is a correct iteration of $\P|\eta$ of length $\leq \d$ then $\Q(\T)\insegeq \M_{n-1}(\c(\T))$ implying that the correct branch of $\T$ is in $\P$. 
\item We will need the following lemma. We continue with the $\P$ and $\d$ of clause 12, but the results we state are more general. 
\begin{lemma}\label{fbcut is weak fb cut} Suppose $\k$ is an $fb$-cut. Then $\k$ is a weak $fb$-cut.
\end{lemma}
\begin{proof} If $\Q$ is an ${\sf{Le}}^{\P|\k}$-premouse which is constructed by the fully backgrounded construction of $\P|\d$ done over ${\sf{Le}}^{\P|\k}$ then by universality $\Q\insegeq {\sf{Le}}^{\P|\d}$ (see clause 11 above). Since $\k$ is an $fb$-cut, we have that $\rho_{\omega}(\Q)\geq \k$. Hence, $\k$ is a weak $fb$-cut.
\end{proof}
\item We will need the following lemma. We continue with the $\P$ and $\d$ as in clause 12.
\begin{lemma}\label{coherence of fb} Let $E\in \vec{E}^{\P|\d}$ be a total extender such that $\nu(E)$ is an inaccessible cardinal of $\P$. Then for any $\tau<\nu(E)$,  $\tau$ is a weak $fb$-cut in $\P$ if and only if $\tau$ is a weak $fb$-cut in $Ult(\P, E)$.
\end{lemma}
\begin{proof} Assume first that $\tau$ is an inaccessible weak $fb$-cut in $\P$. Let $\K={\sf{Le}}^{\P|\tau}$. Suppose $\Q$ is a $\K$-premouse constructed by the fully backgrounded construction of $Ult(\P, E)|\d$ done over $\K$ and $\rho_\omega(\Q)=\tau$\footnote{Recall that if $\Q$ is a $\K$-premouse then we put all elements of $\K$ in all fine structural hulls. Also, our $\Q$ can also be considered as a premouse.}. Since $\Q$ is $\d+1$-iterable, universality implies that $\Q$ is constructed by the fully backgrounded construction of $\P|\d$ done over $\K$. Hence, considering $\Q$ as a premouse, $\rho_\omega(\Q)\geq \tau$. Thus, $\tau$ is a weak $fb$-cut in $Ult(\P, E)$.

Conversely, suppose $\tau$ is an inaccessible weak $fb$-cut in $Ult(\P, E)$. Again, let $\K={\sf{Le}}^{\P|\tau}$. Suppose $\Q$ is a $\K$-premouse constructed by the fully backgrounded construction of $\P|\d$ done over $\K$ and $\rho_\omega(\Q)=\tau$. Then $\Q$ is $\d+1$-iterable in $Ult(\P, E)$\footnote{This is because if $\T$ is a correct iteration of $\Q$ above $\tau$ then $Ult(\P, E)$ can find the correct branch of $\T$ as the function $\T\mapsto (\powerset(\d(\T))\cap \M_{n-1}(\c(\T)))$ is definable over $Ult(\P, E)|(\d^+)^\P$.}. Hence, again universality implies that, considering $\Q$ as just a premouse, $\rho_\omega(\Q)\geq \tau$.
\end{proof}

 \item $\sf{Cond}$ is the following statement in the language of premice.  $\dot{\V}$ is used for the universe. \\\\
${\sf{Cond:}}$ Suppose $\d$ is the least Woodin cardinal, $\k<\d$ is the least $<\d$-strong cardinal (as witnessed by the extender sequence), $\gg>\d$ is an inaccessible cardinal, $\phi$ is a formula and $\dot{\V}|\gg\models \phi[\d, \vec{a}]$ where $\vec{a}\in [\dot{\V}|\d]^{<\omega}$. There is then $\a<\b<\k$ and $\vec{b}\in [\dot{\V}|\a]^{<\omega}$ such that 
\begin{enumerate}
\item $\dot{\V}|\b\models ``\a$ is the least Woodin cardinal",
\item $\vec{b}\in [\dot{\V}|\a]^{<\omega}$,
\item $\dot{\V}|\b\models \phi[\a, \vec{b}]$, and
\item $\a$ is an inaccessible cardinal of $\dot{\V}$,
\item $\a$ is an $fb$-cut of $\dot{\V}$.
\end{enumerate}

We will need the following lemma. $\P$ and $\d$ are as in clause 12 above. Thus, $\P=\M_n(x)$ where $n>0$ ,  $x\in \bR$, and $\d$ is the least Woodin cardinal of $\P$.
\begin{lemma}\label{getting an fbcut} $\P\models {\sf{Cond}}$.
\end{lemma}
\begin{proof} Let $\N={\sf{Le}}^{\P|\d}$. Let $\d$ be the least Woodin cardinal of $\P$. Suppose $\gg$, $\phi$ and $a\in \P|\d$ are such that $\gg>\d$ is inaccessible and $\P|\gg\models \phi[\d, a]$. We can find $\Q\insegeq \P|\d$ and an elementary $\pi:\Q\rightarrow \P|\gg$ such that if $\tau=\pi^{-1}(\d)$ then $\tau$ is an inaccessible cardinal of $\P$, $\pi\rest \tau=id$ and $a\in \P|\tau$. We then have that $\tau$ is an $fb$-cut of $\P$. However, we do not know that $\tau<\k$ where $\k$ is the least $<\d$-strong cardinal of $\P$. 

We now want to show that there is $\R\insegeq \P|\k$ such that 
\begin{enumerate}
\item $\R$ has a Woodin cardinal and if $\nu$ is its least Woodin cardinal then $\nu$ is an inaccessible cardinal of $\P|\k$,
\item $\R|\nu=\P|\nu$ and $\nu$ is an $fb$-cut,
\item for some $b\in \R|\nu$, $\R\models \phi[\nu, b]$.
\end{enumerate}

To get such an $\R\insegeq \P|\l$, let $E\in \vec{E}^\P$ be such that $\Q\insegeq \P|\lh(E)$, $\cp(E)=\k$ and $\tau$ is a cutpoint in $Ult(\P, E)$. Then in $Ult(\P, E)$, $\Q$ has the properties that we look for except that we do not know that $\tau$ is an $fb$-cut in $Ult(\P, E)$. We now prove that in fact $\tau$ is an $fb$-cut in $Ult(\P, E)$.

It follows from \rlem{coherence of fb} that $\tau$ is a weak $fb$-cut in $Ult(\P, E)$. To see that $\tau$ is an $fb$-cut notice that the fully backgrounded construction of $Ult(\P, E)$ never adds extenders overlapping $\tau$, and so in fact this construction is the fully backgrounded construction of $Ult(\P, E)$ done over $\N|\tau$.
\end{proof}
\item We will need the following two extender sequences. 
 \begin{definition}\label{appropriate sequence} Suppose $\S$ is a translatable premouse (see clause 3 above) and $\tau$ is its least Woodin cardinal. Let $\N={\sf{Le}}^{\S|\tau}$. We then let $\mathcal{E}^\S_{le}$ be the set of all extenders $E\in \vec{E}^{\S|\tau}$ such that 
 \begin{enumerate}
 \item $\nu(E)$ is an inaccessible cardinal of $\S$,
 \item $\pi_E(\N)|\nu(E)=\N|\nu(E)$,
 \item $\cp(E)$ is an $fb$-cut.
 \end{enumerate}
 
 We let $\mathcal{E}_{sm}^\S$ be the set of $E\in \vec{E}^{\S|\tau}$ such that $\nu(E)$ is a measurable cardinal of $\S$ and $\cp(E)$ is a $<\tau$-strong cardinal of $\S$. 
 \end{definition}
 \begin{lemma}\label{e is appropriate} Continuing with $\S$ and $\tau$ as above, 
$\mathcal{E}_{ms}^\S$ is weakly appropriate and $\mathcal{E}^\S_{le}$ is appropriate (see \rdef{wap}).
 \end{lemma}
 \begin{proof} It is easy to show that $\mathcal{E}_{sm}^\S$ is weakly appropriate. Suppose then $E\in \mathcal{E}^\S_{le}$. Let $\N$ be the fully backgrounded construction of $\S|\tau$ and set $\mathcal{E}=\mathcal{E}^\S_{le}$. We want to see that $\pi_E(\mathcal{E})\cap (\S|\nu(E))=\mathcal{E}\cap (\S|\nu(E))$. 
 We have that $\pi_E(\N)$ is the fully backgrounded construction of $Ult(\S, E)|\tau$ and $\pi_E(\N)|\nu(E)=\N|\nu(E)$. 
 
 Suppose now that $F\in \pi_E(\mathcal{E})\cap (\S|\nu(E))$. It follows that\\\\
(1) $\nu(F)$ is an inaccessible cardinal in $Ult(\S, E)$,\\
(2) $\pi^{Ult(\S, E)}_F(\pi_E(\N))|\nu(F)=\pi_E(\N)|\nu(F)$,\\
 (3) $\cp(F)$ is an $fb$-cut in $Ult(\S, E)$.\\\\
 We want to see that \\\\
(4) $\nu(F)$ is an inaccessible cardinal in $\S$,\\
(5) $\pi_F(\N)|\nu(F)=\N|\nu(F)$,\\
(6) $\cp(F)$ is an $fb$-cut in $\S$.\\\\
(4) is easy as $\nu(E)>\nu(F)$ and $\nu(E)$ is inaccessible in both $\S$ and $Ult(\S, E)$. To see (5) notice that because \begin{center}$\pi_F\rest (\S|(\cp(F)^+)^\S)=\pi_F^{Ult(\S, F)}\rest (Ult(\S, F)|(\cp(F)^+)^\S)$\end{center} and because of (2) we have that $\N|\nu(F)=\pi_F(\N)|\nu(F)$.

We now want to show (6) and we have that $\cp(F)$ is an $fb$-cut in $Ult(\S, F)$. It follows that $\pi_E(\N)|\cp(F)$ is the fully backgrounded construction of $Ult(\S, F)|\cp(F)$. But $\pi_E(\N)|\cp(F)=\N|\cp(F)$ and $Ult(\S, F)|\cp(F)=\S|\cp(F)$. Therefore, 
\begin{center}$\N|\cp(F)$ is the fully backgrounded construction of $\S|\cp(F)$. \end{center}
The proof that $\mathcal{E}\cap (\S|\nu(E))\subseteq \pi_E(\mathcal{E})\cap (\S|\nu(E))$ is very similar.
 \end{proof}
 The same proof can be used to show that the following also holds.

 \begin{lemma}\label{e is appropriateA} Continuing with $\S$ and $\tau$ as above, 
if $E\in \mathcal{E}^\S_{le}$ then for all $\xi<\nu(E)$, $\S\models ``\xi$ is an $fb$-cut" if and only if $Ult(\S, E)\models ``\xi$ is an $fb$-cut". 
 \end{lemma}
 \begin{proof} Let $\N$ be as in the above proof. We have that $\pi_E(\N)|\nu(E)=\N|\nu(E)$. Set $\K=\sf{Le}^{\S|\xi}$. We have that $\K=\sf{Le}^{Ult(\S, E)|\xi}$. Thus, the following equivalence holds.
 \begin{align*}
\S\models ``\xi\ \text{is an} fb\text{-cut}" &\iff \K=\N|\xi \\ &\iff \K=\pi_E(\N)|\xi \\
&\iff Ult(\S, E)\models ``\xi\ \text{is an} fb\text{-cut}".  
\end{align*}
 \end{proof}
 The following corollary can now easily be proven by induction on the length of the iteration. 
 
 \begin{corollary}\label{cor to above} Continuing with $\S$ and $\tau$ as above, suppose $\T$ is a normal non-dropping (i.e. $\mathcal{D}^\T=\emptyset$) iteration of $\S$ such that for every $\a<\lh(\T)$, $E_\a^\T\in \pi_{0, \a}^\T(\mathcal{E}^\S_{le})$. Then for any $\a<\b<\lh(\T)$, $\M^\T_\b\models ``\cp(E_\a^\T)$ is an $fb$-cut".
 \end{corollary}
\item We will need the following lemma.
\begin{lemma}\label{generic over le} Suppose $\S$ is a premouse and $\tau$ is its least Woodin cardinal. Let $\N={\sf{Le}}^{\S|\tau}$ and assume that ${\sf{Ord}}\cap \N=\tau$. Suppose $\a<\b$ are such that 
\begin{enumerate}
\item $\a$ is an $fb$-cut in $\S$ and
\item $\S|\b\models {\sf{ZFC}}+``\a$ is a Woodin cardinal".
\end{enumerate}
Suppose $t$ is a real which satisfies all the axioms of ${\sf{Ea}}^\S_{\tau, \mathcal{E}^\S_{le}}$ that are generated by the extenders in $\S|\a$. Let $\K={\sf{StrLe}}(\S|\b)$. Then $t$ is generic for ${\sf{Ea}}^{\K}_{\a, \mathcal{E}^\K_{sm}}$.
\end{lemma}
\begin{proof} Suppose $E\in \vec{E}^{\K|\a}$ such that $\nu(E)$ is a measurable cardinal of $\K$ and $\cp(E)$ is a $<\a$-strong cardinal of $\K$. Let $F$ be the background extender of $E$. It follows that $E=F\cap \K$. Moreover, since $\cp(E)$ is a strong cardinal in $\K|\a$, $\cp(E)$ is an $fb$-cut in $\S|\a$, and since $\a$ itself is an $fb$-cut, $\cp(E)$ is an $fb$-cut in $\S$. Also, $F$, since it coheres $\N$, is in $\mathcal{E}^\S_{le}$. Therefore, since there is a factor map $k:\pi_E^\N(\N|(\cp(E)^+)^\N)\rightarrow \pi_F^\S(\N|(\cp(E)^+)^\N)$ with $\cp(k)\geq \nu(E)$ and since $\pi_E^\N(\N|(\cp(E)^+)^\N)|\nu(E)=\N|\nu(E)$, any axiom generated by $E$ in $\K$ is satisfies by $z$ as it is also an axiom generated by $F$ (see clause 2 of \rrev{s constructions}). 
\end{proof}
\end{enumerate}

\begin{review}\label{sargpre}
\textbf{Review of \cite{SargPre}:}
\end{review}
Unless otherwise specified, we assume ${\sf{AD}}^{L(\bR)}$. The material reviewed below appears in \cite{SargPre} and in \cite{HODCoreModel}. Other treatments of similar concepts appear in \cite{Varsovian} and \cite{Sandra}.
\begin{enumerate}
\item $(u_i: i\in {\sf{Ord}})$ is the sequence of uniform indiscernibles for reals (assuming they exist).\footnote{Assume that every real has a sharp, and for each $x\in \bR$, let $C_x$ be the class of $x$-indiscernibles. Thus, $C_x$ is a class of indiscernibles for $L[x]$. We then say that $u$ is a uniform indiscernible if $u\in \cap_{x\in \bR}C_x$. Thus, the class of uniform indiscernibles is a club and contains all uncountable cardinals.} Set $s_{0}=\emptyset$ and for $m\geq 1$, $s_m=(u_0, ..., u_{m-1})$
\item Fix $n\in \omega$ and $x\in \bR$ and suppose $\P$ is a complete iterate of $\M_{n}(x)$. For $i\leq n$, $\d_i^\P$ is the $i+1$st Woodin of $\P$. Let 
\begin{center}
$\gg_m^\P=\sup(Hull_1^{\P|u_{m}}(s_m)\cap \d_0^\P)$.
\end{center}
Then $\sup_{m<\omega}\gg_m^\P=\d_0^\P$. Given two pairs $(\P, \a)$ and $(\Q, \b)$ such that $\P$ and $\Q$ are complete iterates of $\M_n(x)$, $\a<\gg_m^\P$ and $\b<\gg^\Q_m$, we write $(\P, \a)\leq^{n}_m (\Q, \b)$ if and only if letting $\R$ be the common iterate of $\P$ and $\Q$, $\pi_{\P, \R}(\a)\leq \pi_{\Q, \R}(\b)$. Clearly, $\leq^n_m$ depends on $x$, but not mentioning $x$ makes the notation simpler\footnote{But $\leq^n_m$ does not depend on the choice of indiscernibles.}.  We then have that there is a formula $\phi$ such that for all 
$(\P, \a)$ and $(\Q, \b)$, $(\P, \a)\leq^{n}_m (\Q, \b)$ if and only if
\begin{center}
$ \M_{n-1}(\M_n(x)^\#, (\P, \a), (\Q, \b))\models \phi[\M_n(x)^\#, (\P, \a), (\Q, \b), s_m]$.
\end{center}
The formula $\phi$ essentially records what was said above. It is the conjunction of the following statements:
\begin{enumerate}
\item $\P$ and $\Q$ are complete iterates of $\M_n(x)$.
\item There is $\R$, which is a common iterate of $\P$ and $\Q$ and $\pi_{\P, \R}(\a)\leq \pi_{\Q, \R}(\b)$.
\end{enumerate}
Explaining the above equivalence is beyond the scope of this paper. The reader can consult \cite{SargPre}, \cite{HODCoreModel}, \cite{Varsovian} and \cite{Sandra}. Essentially, the following happens. 

Suppose $\P$ is a complete iterate of $\M_n(x)$. Then, there is an iteration tree $\T\in \M_{n-1}(\M_n(x)^\#, (\P, \a), (\Q, \b))$ such that $\T$ is an iteration tree on $\M_{n}(x)^\#$ according to its unique $(\omega_1+1$-iteration strategy and such that either $\T$ has a last model which is $\P$ or $\P=L(\c(\T))$ and $\d(\T)$ is the largest Woodin cardinal of $\P$ (such iteration trees are called \textit{maximal}). This is the fact used to identify complete iterates of $\M_n(x)^\#$ inside models of the form $\M_{n-1}(\M_n(x)^\#, d)$. 

Suppose next that there is $\R$ such that $\pi_{\P, \R}(\a)\leq \pi_{\Q, \R}(\b)$. We can assume that $\R$ is obtained via the usual comparison procedure of \cite{OIMT}, which means that $\R\in \M_{n-1}(\M_n(x)^\#, (\P, \a), (\Q, \b))$. One then shows that $$\M_{n-1}(\M_n(x)^\#, (\P, \a), (\Q, \b))$$ can correctly compute $\pi_{\P, \R}\rest \gamma^\P_m$ and $\pi_{\Q, \R}\rest \gamma_m^\Q$, even though it cannot in general compute $\pi_{\P, \R}$ and $\pi_{\Q, \R}$.

\item For $x\in \bR$ and $n\in \omega$, set $\gg^{2n+1}_{m, x, \infty}=\pi_{\M_{2n+1}(x), \infty}(\gg_m^{\M_{2n+1}(x)})$.
\item For $x\in \bR$ and $n\in \omega$, set $b_{2n+1, m}=\sup_{x\in \bR}\gg_{m, x, \infty}^{2n+1}$.
\item Let $\k^1_{2n+1}$ be the predecessor of $\d^1_{2n+1}$ (for example see \cite[Theorem 2.18]{Jackson}). It follows from \cite{SargPre} that for each $n$, $\sup_{m\in \omega}b_{2n+1, m}=\k^1_{2n+3}$. In fact, for each $x\in \bR$, $\sup_{m<\omega}\gg_{m, x, \infty}^{2n+1}=\k^1_{2n+3}$. Moreover, for each $n, m\in \omega$, $b_{2n+1, m}$ is a cardinal and $b_{2n+1, 0}>\d^1_{2n+1}$. Thus, $b_{2n+1, 0}\geq \d^1_{2n+2}$. For the proofs of these results see \cite[Theorem 4.1, Corollary 5.23, Lemma 6.1]{SargPre}.
\item It is conjectured in \cite{SargPre} that for all $n$, $b_{2n+1, 0}=\d^1_{2n+2}$. For $n=0$ this is shown in \cite{BoundHjorth}. The case $n\geq 1$ is still open.
\end{enumerate}

\begin{review}\label{gamma stable} \textbf{$\gg$-stability:}
\end{review}
The main technical fact from \cite{SargPre} that we will need in this paper appears in the bottom of page 760 of  \cite{SargPre}. It claims that for each $x\in \bR$, each $n\in \omega$ and for each $\gg<\k^1_{2n+3}$, there is a  $\gg$-stable complete iterate $\P$ of $\M_{2n+1}(x)$. Because our situation is just a little bit different we outline how to obtain $\gg$-stable iterates. The reader may find it useful as well.\\\\ 
\textbf{$\gg$-stable iterates:}
\begin{enumerate}
\item Fix $n\in \omega$ and $x\in \bR$. Let $\M=\M_{2n+1}(x)$ and suppose that $\M^\#_{2n+1}\in \M$\footnote{Any reader who is familiar with the nuts and bolts of inner model theory can see that this condition is not necessary.}. Suppose $\gg$ is an ordinal such that for some $m_0$, $\gg<\gg^{2n+1}_{m_0, \infty}$. Let $\P$ be a complete iterate of $\M$.
\item Let $\nu$ be the least inaccessible of $\P$ above $\d_0^\P$ and let $\mH^\P$ be the direct limit of all iterates of $\M_{2n+1}$ that are in $\P|\nu$\footnote{We say $\Q$ is an iterate of $\M_{2n+1}$ in $\P|\nu$ if $\Q$ is a complete iterate of $\M_{2n+1}$ and letting $\d$ be the largest Woodin cardinal of $\Q$, $\Q|\d\in \P|\nu$.}. The construction of such limits has been carried out in \cite{SargPre}, \cite{Sandra}, \cite{Varsovian}. For example see \cite[Section 5.1]{SargPre}.
\item The results of \cite{SargPre} and other similar calculations done for example in \cite{HODCoreModel} show that $\mH^\P$ is a complete iterate of $\M_{2n+1}$. 
\item We say $\P$ is locally $\gg$-stable if $\gg\in \rge(\pi_{\mH^\P, \infty})$.
\item We say $\P$ is $\gg$-stable if $\P$ is locally $\gg$-stable and if $\xi=\pi_{\mH^\P, \infty}^{-1}(\gg)$ then whenever $\Q$ is a complete iterate of $\P$, 
\begin{center}
$\pi_{\mH^\Q, \infty}(\pi_{\P, \Q}(\xi))=\gg$.
\end{center}
\item The argument on page 760 of  \cite{SargPre} can be used to show the following lemma.  
\begin{lemma}\label{stable p} There is a complete $\gg$-stable iterate $\P$ of $\M_{2n+1}(x)$.
\end{lemma}
\begin{proof}
(Outline) Towards a contradiction assume not. The key observation is that whenever $\P$ is a complete iterate of $\M_{2n+1}(x)$ and $\Q$ is a complete iterate of $\P$, the Dodd-Jensen argument (see \cite[Chapter 4.2]{OIMT}) implies that for all $\xi\in \mH^\P$, $\pi_{\mH^\P, \infty}(\xi)\leq \pi_{\mH^\Q, \infty}(\pi_{\P, \Q}(\xi))$. 

Suppose now that $\R$ is a complete iterate of $\M_{2n+1}$ such that $\gg\in \rge(\pi_{\R, \infty})$. We  set $\gg_\R=\pi^{-1}_{\R, \infty}(\gg)$. 

Let now $\P$ be a compete iterate of $\M_{2n+1}(x)$ and $\Q$ be a complete iterate of $\P$. What we have observed implies that if $\gg_{\mH^\P}$ and $\gg_{\mH^\Q}$ are defined then $\pi_{\P, \Q}(\gg_{\mH^\P})\geq \gg_{\mH^\Q}$. Moreover, if $\Q$ witnesses that $\P$ is not $\gg$-stable then we in fact have that $\pi_{\P, \Q}(\gg_{\mH^\P})>\gg_{\mH^\Q}$.

 Observe further that given any $(\P, \Q)$  we can iterate $\Q$ to obtain a complete iterate $\R$ of $\Q$ such that $\gg_{\mH^\R}$ is defined. To do this, it is enough to fix some complete iterate $\S$ of $\M_{2n+1}$ for which $\gg_\S$ is defined and iterate $\Q$ to obtain a complete iterate $\R$ of $\Q$ such that $\S$ is generic for $\sf{Ea}^{\R}_{\d_0^\R}$.\footnote{According to our conventions, this means that the extender algebra over which $\S$ is generic has $\l$ many generators where $\l$ is the largest Woodin cardinal of $\S$.} It then follows that $\S$ is a point in the directed system of $\R[\S]$ that converges to $\mH^\R$, and so $\gg_{\mH^\R}$ is defined.
 
 We now get that our assumption that there is no $\gg$-stable complete iterate of $\M$ implies that there is a sequence $(\P_i: i<\omega)$ such that (a) for each $i$, $\P_{i+1}$ is a complete iterate of $\P_i$, (b) for each $i$, $\gg_{\mH^{\P_i}}$ is defined (c) $\P_0$ is a complete iterate of $\M$ and finally (d) for each $i$, 
\begin{center}
$\pi_{\P_i, \P_{i+1}}(\gg_{\mH^{\P_i}})>\gg_{\mH^{\P_{i+1}}}$. 
\end{center}
We thus have that the direct limit of $(\P_i: i<\omega)$ is ill-founded, which is clearly a contradiction.
\end{proof}
\item Suppose now that $\P$ is $\gg$-stable, $\nu<\d_0^\P$ and $p\in {\sf{Ea}}^\P_{\d_0^\P, \nu}$. Suppose $m\leq k$ are natural numbers (so $m\geq 1$). We say $p$ is $(\gg, k, m)$-good if $p$ forces that when the generic is decoded as a $k$-tuple $(u_1, ..., u_k)$ (via one of the standard ways of decoding a real into tuples) $u_m$ codes a complete iterate $\R$ of $\M_{2n+1}$ with the property that $\T_{\M_{2n+1}, \R}$ is below the least Woodin cardinal of $\M_{2n+1}$ and $\gg_{\mH^\P}\in \rge(\pi_{\R, \mH^\P})$. 
\item It is not immediately clear that the definition of $(\gg, k, m)$-good condition is first order over $\P$. This follows from the fact that the directed system can be internalized to $\P[g]$ where $g\subseteq Coll(\omega, \d_0^\P)$ is $\P$-generic. To do this, one uses the concept of $s_n$-iterability and the details of this were carried out in \cite{SargPre}. For example, see \cite[Definition 5.4]{SargPre} and \cite[Chapter 5.1]{SargPre}. The basic ideas that are used in internalizing directed systems are due to Woodin who carried out such internalization in $L[x]$ (see \cite{HODCoreModel}).
\item Because the statement ``$p$ is $(\gg, k,m)$-good" is first order over $\P$, there is a maximal antichain of $(\gg, k, m)$-good conditions. As ${\sf{Ea}}^\P_{\d_0^\P, \nu}$ has $\d_0^\P$-cc, we have a condition $p$ such that $p$ is the disjunction of the conditions of some maximal antichain consisting of $(\gg, k, m)$-good conditions, and as a consequence, $p$ is $(\gg, k, m)$-good and if $g\subseteq {\sf{Ea}}^\P_{\d_0^\P, \nu}$ is generic such that $g$ is a tuple $(g_1, g_2,..., g_k)$ and some $r\in g$ is  $(\gg, k, m)$-good then $p\in g$. We can then let $p_{\gg, \nu, k, m}^\P$ be the $\P$-least $(\gg, k, m)$-good condition with the property that any other $(\gg, k, m)$-good condition is compatible with it. We call $p_{\gg, \nu, k, m}^\P$ the $(\gg, \nu, k, m)$-master condition.
\end{enumerate}

\section{$\Pi^1_n$-iterability}\label{pi iterability} 

 We will use the concept of $\Pi^1_n$-iterability introduced in \cite[Definition 1.4]{PWIM}. Following \cite[Definition ]{PWIM}, we say that a premouse $\M$ is $n$-small if for every $\a\in \dom(\vec{E})$,\footnote{Recall that $\vec{E}$ is a partial function with domain some ordinal $\gg$. $\vec{E}(\a)$ is the extender indexed at $\a$.} $\M|\a\models ``$there does not exist $n$ Woodin cardinals". Thus, $\M_1^\#$ is 2-small and $\M_1$ is 1-small while $\M_2^\#$ is 3-small and $\M_2$ is 2-small. Notice however that all proper initial segments of $\M_1^\#$ are 1-small. We then say that $\M$ is properly $n$-small if $\M$ is $n+1$-small and all of its proper initial segments are $n$-small.
 
  We will use $\Pi^1_{2n+2}$-iterability in conjunction with properly $2n+1$-small premice as well as $2n+1$-small premice. There is a lot to review here, but the details are technical and we suggest that the reader consult \cite{PWIM}. The following is a list of facts that we will need.
 \begin{enumerate} 
 \item We say  $\M$ is $\M_n$-like premouse over $x$ if $\M$ is $n$-small premouse over $x$ with $n$ Woodin cardinals and such that ${\sf{Ord}}\subseteq \M$. Similarly, we say $\M$ is $\M_n^\#$-like premouse over $x$ (or just $x$-premouse) if $\M$ is active, sound, properly $n$-small premouse over $x$ with $n$ Woodin cardinals.  
 \item Given two premice $\M$ and $\M'$ we say $(\T, \T')$ is a coiteration of $(\M, \M')$ if the extenders of $\T$ and $\T'$ are chosen to be the least ones causing disagreement. More precisely, $\lh(\T)=\lh(\T')$ and for every $\a<\lh(\T)$, $\lh(E_\a^\T)$ is the least $\b$ such that $\M_\a^\T|\b=\M_\a^{\T'}|\b$ but $\M_\a^\T||\b\not=\M_\a^{\T'}||\b$, and similarly for $E_\a^{\T'}$. We then must have that $\lh(E_\a^\T)=\lh(E_\a^{\T'})$. It is possible that only one of $E_\a^\T$ and $E_\a^{\T'}$ is defined, in which case we allow padding\footnote{We didn't have to abuse our terminology this way if we didn't insist on having $\lh(\T)=\lh(\T')$, but this simplifies the notation.}. 
 \item We say that the coiteration $(\T, \T')$ of $(\M, \M')$ is successful if $\lh(\T)$ is a successor ordinal and if $\N$ is the last model of $\T$ and $\N'$ is the last model of $\T'$ then either $\N\insegeq \N'$ or $\N'\insegeq \N$. 
 \item Suppose $\M$ is a premouse, $\T$ is an iteration of $\M$ of limit length, $\a<\omega_1$ and $b$ is a maximal branch of $\T$. We say $\N$ is $(\a, b)$-relevant if either $\N=\M^\T_b$ or for some $\P\insegeq \M^\T_b$ and for some $E\in \vec{E}^\P$, $\N$ is the $\a$th linear iterate of $\P$ via $E$. We say that $b$ is $\a$-good if whenever $\N$ is $(\a, b)$-relevant, either $\N$ is well-founded or $\a$ is in the well-founded part of $\N$. We say $\M$ is $\Pi^1_2$-iterable if for every iteration $\T$ of $\M$ of countable limit length and for every ordinal $\a$ there is an $\a$-good maximal branch $b$ of $\T$. 
 \item Recall $\mathcal{G}(\M, 0, 2n+1)$, which is the iteration game introduced on page 83 of \cite{PWIM}. $\mathcal{G}(\M, 0, 1)$  defines $\Pi^1_2$-iterability. In $\mathcal{G}(\M, 0, 1)$, Player I plays a pair $(\T, \b)$ such that $\T$ is an iteration of $\M$ and $\b<\omega_1$. Player II must either accept $\T$ or play a $\b$-good branch of $\T$. If $\lh(\T)$ is a limit ordinal or $\T$ has an ill-founded last model then Player II is not allowed to accept $\T$. If Player II doesn't violate any of the rules of the game then Player II wins the game. We thus have that $\M$ is $\Pi^1_2$-iterable if and only if Player II has a winning strategy in $\mathcal{G}(\M, 0, 1)$.
 \item Suppose $\M$ is as above. We say that $\M$ is $\Pi^1_{2n+2}$-iterable if player $II$ has a winning strategy in $\mathcal{G}(\M, 0, 2n+1)$. The game has $2n+1$ many rounds. We describe the game for $n>0$. We modify the game very slightly and present it as a game that has only two rounds. We start by assuming that $\M$ is a $\d$-mouse in the sense of \cite[Definition 1.3]{PWIM}\footnote{$\d$-mouse is also sometimes called $\d$-sound.}. 
 \begin{enumerate}
 \item In the first round, Player I plays a pair $(\T_0, x_0)$ such that $\T_0$ is an iteration of $\M$ that is above $\d+1$\footnote{For each $\a<\lh(\T)$, $\cp(E_\a)>\d$.} and $x_0\in \bR$. Player II must either accept $\T_0$ or play a maximal well-founded branch $b_0$ of $\T_0$. If $\lh(\T_0)$ is a limit ordinal or $\T_0$ has an ill-founded last model then Player II is not allowed to accept $\T_0$. If Player II accepts $\T_0$ then we set $\d_1=\sup\{\nu(E_\a^{\T_0}): \a+1<\lh(\T_0)\}$ and $\N_1=\M_{\lh(\T_0)-1}^{\T_0}$. If Player II plays $b$ then we set $\d_1=\d(\T_0)$ and $\N_1=\M^{\T_0}_{b_0}$.
 \item The second round is played on $\N_1$, which is now a $\d_1$-mouse. In this round, Player I plays an  iteration $\T_1$ of $\N_1$ above $\d_1$ such that for some $\Pi^1_{2n}$-iterable $\M_{2n-1}^\#$-like $(\T_0, b_0, x_0)$-premouse $\M'$, $\T_1\in \M'|\omega_1^{\M'}$\footnote{The requirement that $\T_1\in \M'|\omega_1^{\M'}$ is similar to the  requirement that $\T_1\in \Delta^{{\sf{HC}}}_{2n}(\T_0, b_0, x_0)$ used in \cite{PWIM}. It all comes down to the fact that for $\vec{s}\in \bR^{<\omega}$, $\exists u \in \Q_{2n+1}(\vec{s})$ is $\Pi^1_{2n+1}(\vec{s})$ uniformly in $\vec{s}$. For this and similar results one can consult \cite{Qtheory} or \cite[Corollary 4.9]{PWIM}.}. Once again Player II must either accept $\T_1$ or play a maximal well-founded branch $b_1$ of $\T_1$. If $\lh(\T_1)$ is a limit ordinal or $\T_1$ has an ill-founded last model then Player II is not allowed to accept $\T_1$. If Player II accepts $\T_1$ then we set $\d_2=\sup\{\nu(E_\a^{\T_1}): \a+1<\lh(\T_1)\}$ and $\N_2=\M_{\lh(\T_1)-1}^{\T_1}$. If Player II plays $b_1$ then we set $\d_2=\d(\T_1)$ and $\N_2=\M^{\T_1}_{b_1}$.
 \end{enumerate}
 Now, Player II wins if Player II doesn't violate any of the rules of the game and $\N_2$ is $\Pi^1_{2n}$-iterable as a $\d_2$-mouse. 
 \item For $n>1$, the statement $``$the real $x$ codes a  $\Pi^1_n$-iterable premouse" is $\Pi^1_n$.
 \item Assuming that $\Pi^1_{2n}$-iterability is a $\Pi^1_{2n}$-condition, it is not hard to see that $\Pi^1_{2n+2}$-iterability is a $\Pi^1_{2n+2}$-condition. \footnote{The proof appears in \cite{PWIM}.}
  \item Suppose $x\in \bR$, $n\in \omega$ and $\M$ is $\M_{2n+1}^\#$-like $\Pi^1_{2n+2}$-iterable $x$-premouse.  Let $\d$ be the least Woodin cardinal of $\M$ and suppose $\T$ is an iteration of $\M$ that is below $\d$. We say $\T$ is correct if $\lh(\T)\leq \omega_1$ and for each limit $\a<\lh(\T)$ if $b_\a=[0, \a)_\T$ then $\Q(b_\a, \T\rest \a)$ exists and $\Q(b_\a, \T\rest \a)\insegeq \M_{2n}(\c(\T\rest \a))$. 
  \item \cite{PWIM} shows that for $x\in \bR$, $\M_n^\#(x)$ is the unique $\Pi^1_{n+2}$-iterable, active, properly $n$-small,  sound premouse over $x$. For example see clause 3 of \cite[Lemma 2.2]{PWIM}, and also the remark after \cite[Corollary 4.11]{PWIM}.
\item It follows from \cite{PWIM} that if $\M$ is $2n$-small, $\Pi^1_{2n+2}$-iterable, sound $\d$-mouse (in the sense of \cite[Definition 1.3]{PWIM}) over a real $x$ then $\M$ is $\omega_1$-iterable above $\d$. For example, see \cite[Lemma 3.3]{PWIM}.
\item \cite[Corollary 4.7]{PWIM} implies that for all $n\in \omega$ and $x\in \bR$, $\M_{2n}(x)$ is $\Sigma^1_{2n+2}$-correct. More precisely, if $\vec{a}\in \bR^{<\omega}\cap \M_{2n}(x)$ and $\phi$ is a $\Sigma^1_{2n+2}$ formula then $\M_{2n}(x)\models \phi[\vec{a}]$ if and only if $V\models \phi[\vec{a}]$. 
\item The following is a consequence of clause 12 above and the proof of \rlem{correctness}.
\begin{proposition}\label{correctness1} Suppose $x\in \bR$ and $\M$ is a complete iterate of $\M_{2n}(x)$ such that $\T_{\M_{2n}(x), \M}$ has a countable length. Let $\exists u \phi(\vec{v})$ be a $\Sigma^1_{2n+2}$ formula. Let $\d_0<\d_1<...<\d_{2n-1}$ be the Woodin cardinals of $\M$. Set $d=(\d_0, ..., \d_{2n-1})$. Suppose $\vec{a}\in \M\cap \bR^k$ is such that $V\models \exists u \phi[\vec{a}]$. There is then $b\in \bR\cap \M_{2n}(x)$ such that $\M\models \phi_{\M, d}[b, \vec{a}]$. 
\end{proposition}
\end{enumerate}
  
      The following are consequences of \cite{PWIM} that we will need. The proofs of these facts are implicitly present in \cite{PWIM}, but \cite{PWIM} doesn't isolate these facts.
  
    \begin{lemma}\label{cofinal branch}
  Suppose $\M$ is $\Pi^1_2$-iterable and $\M_1^\#$-like, and suppose $\T$ is a correct iteration of $\M$ of countable limit length. Let $\b$ be such that for each limit $\a<\lh(\T)$, if $b_\a=[0, \a)_\T$ then $\Q(b_\a, \T\rest \a)$ exists and $\Q(b_\a, \T\rest \a)\inseg \mathcal{J}_\b(\c(\T\rest \a))$. Suppose $I$ plays $(\T, \b)$ in $\mathcal{G}(\M, 0, 1)$ and $II$ responds, using her winning strategy in $\mathcal{G}(\M, 0, 1)$, with a branch $b$. Then $b$ is a cofinal branch of $\T$. 
  \end{lemma}
  \begin{proof} Towards a contradiction suppose that $b$ is not a cofinal branch, and set $\a=\sup(b)$. 
  
  Suppose for a moment that $\Q(b, \T\rest \a)$ exists. We claim that then $\Q(b, \T\rest \a)\insegeq \mathcal{J}(\c(\T\rest \a))$. To see this, suppose not. Because $\mathcal{J}_\b(\c(\T\rest \a))\models ``\d(\T\rest \a)$ is not a Woodin cardinal", we have that $\b$ is not in the well-founded part of $\M^{\T\rest \a}_b$. This implies that $\M^{\T\rest \a}_b$ is well-founded and has ordinal height less than $\b$. But since $\Q(b, \T\rest \a)\not \insegeq \mathcal{J}(\c(\T\rest \a))$, $\Q(b, \T\rest \a)$ must have an active level above $\d(\T \rest \a)$. Let $E\in \vec{E}^{\Q(b, \T\rest \a)}$ be such that $\cp(E)>\d(\T)$. Let $\N$ be the $\b$th iterate of $\Q(b, \T\rest \a)$ via $E$. It follows that $\b$ is in the well-founded part of $\N$, and so $\N|\b\models ``\d(\T\rest \a)$ is not a Woodin cardinal". Hence, $\Q(b, \T\rest \a)\insegeq \N|\b$, which is a contradiction as no mouse is an initial segment of its own iterate.  
  
Now, because the existence of $\Q(b, \T\rest \a)$ implies that  $\Q(b, \T\rest \a)\insegeq \mathcal{J}_\b(\c(\T\rest \a))$, we must have that if $\Q(b, \T\rest \a)$ exists then $\Q(b, \T\rest \a)=\Q(b_\a, \T\rest \a)$ which then implies that $b=b_\a$. Thus, because $b\not =b_\a$, $\Q(b, \T\rest \a)$ does not exist. 
  
It now immediately follows that $\b$ cannot be in the well-founded part of $\M^{\T\rest \a}_b$. Thus, $\M^{\T\rest \a}_b$ must be well-founded and $\sup({\sf{Ord}}\cap \M^{\T\rest \a}_b)< \b$. Because $\Q(b, \T\rest \a)$ does not exist, $\pi^{\T\rest \a}_b$ is defined. If now $\N$ is the $\b$th iterate of $\M^{\T\rest\a}_b$ via the last extender of $\M^{\T\rest \a}_b$, $\N\models ``\d(\T\rest \a)$ is not a Woodin cardinal". This contradicts the fact that $\Q(b, \T\rest \a)$ does not exist.
  \end{proof}
  
  \begin{proposition}\label{simple iterability}
  Assume $V$ is closed under the sharp function $x\mapsto x^\#$. Suppose $\M$ is $\Pi^1_2$-iterable and $\M_1^\#$-like, and suppose $\T$ is a correct iteration of $\M$ of countable limit length. Let $\b$ be such that for each $\a<\lh(\T)$, if $b_\a=[0, \a)_\T$ then $\Q(b_\a, \T\rest \a)$ exists and $\Q(b_\a, \T\rest \a)\inseg \mathcal{J}_\b(\c(\T))$. Suppose that there is an $\a$ such that $\mathcal{J}_\a(\c(\T))\models ``\d(\T)$ is a Woodin cardinal" and $\mathcal{J}_{\a+1}(\c(\T))\models ``\d(\T)$ is not a Woodin cardinal", and finally, suppose $I$ plays $(\T, \max(\a+1, \b))$ in $\mathcal{G}(\M, 0, 1)$ and $II$ responds with a branch $b$ (according to her winning strategy). Then $b$ is the unique cofinal well-founded branch of $\T$ such that $\mathcal{J}_{\a}(\c(\T))\insegeq \M^\T_b$.
  \end{proposition}
  \begin{proof}
We have that \rlem{cofinal branch} implies that $b$ is cofinal. Also, there can be at most one such well-founded cofinal branch. Towards a contradiction assume that either $b$ is ill-founded or $\mathcal{J}_{\a}(\c(\T))\not \insegeq \M^\T_b$.  Suppose for a moment that $b$ is well-founded. Then we must have that $\mathcal{J}_{\a}(\c(\T))\not \insegeq \M^\T_b$, which implies that $\pi^\T_b$ is defined. But now if $\N$ is the $\max(\a+1, \b)$th iterate of $\M^\T_b$ via its last extender then $\mathcal{J}_{\a+1}(\c(\T))\insegeq\N$. Hence, $\N\models ``\d(\T)$ is not a Woodin cardinal" implying that $\mathcal{J}_{\a+1}(\c(\T)) \insegeq \M^\T_b$. 

We thus have that $b$ is an ill-founded branch, and this implies that $\max(\a+1, \b)$ is in the well-founded part of $\M^\T_b$, which then implies that $\Q(b, \T)$ exists and $\Q(b, \T)=\mathcal{J}_{\a}(\c(\T))$. Thus, we have that\\\\
(*) there is a cofinal branch $b$ of $\T$ such that $\Q(b, \T)$ exists, $\Q(b, \T)=\mathcal{J}_{\a}(\c(\T))$ and this branch is ill-founded. \\\\

Let $\M'$ be the result of iterating the last extender of $\M$ out of the universe and let $\gg>\max(\a, \b)$ be such that the ill-foundedness of $b$ can be witnessed by functions in $\M'|\gg$\footnote{Notice that after straighforward re-organization, $\T$ can be viewed as an iteration $\T^*$ of $\M'$. In the presence of sharps, the illfoundedness of $b$ as a branch of $\T$ is equivalent to the illfoundedness of $b$ as a branch of $\T^*$.}. We can now express (*)  by the following formula:\\\\ 
  (**)  there are reals $x, y$ and $z$ such that
  \begin{enumerate}
  \item $x$ codes an iteration $\U$ of $\M'|\gg$
  \item $y$ codes a well-founded model $\N=\mathcal{J}_{\a}(\c(\U))$ such that $J_{1}[\N]\models ``\d(\U)$ is not a Woodin cardinal",
  \item  $z$ codes a cofinal ill-founded branch $c$ of $\U$ such that $\N\insegeq \M^\U_c$.
  \end{enumerate}
  (**) is a $\Sigma^1_2$ statement in any code of $(\M'|\gg, \a)$, and so we get such a $\U$ in $\M[g]$ where $g\subseteq Coll(\omega, \b)$ is generic. But then we can use the proof of \cite[Corollary 4.17]{HODCoreModel} to get a contradiction\footnote{Here is a little bit more for those who are familiar with such proofs. We fix $\theta$ larger than $\gg$ and let $\pi: \P\rightarrow \M'|\theta$ be a countable elementary hull of $\M'|\theta$ inside $\M'$ such that $(\gg, \a)\in \rge(\pi)$. Let then $g\in \M'$ be $\P$-generic for $Coll(\omega, \pi^{-1}(\gg))$. Let $(\T', \N')\in \P[g]$ be a pair satisfying (**). Using Martin-Steel realizability theorem (see \cite[Theorem 4.3]{IT}) we can show that the branch $c'$ of clause 3 of (**) is the $\pi$-realizable branch of $\T'$ and hence, it is well-founded. The fact that $c'\in \P[g]$ follows from absoluteness and the fact that $\N'$ uniquely identifies $c'$.}.
    \end{proof}

\begin{proposition}\label{correct branch} Suppose $n\in \omega$ and $\M$ is $\M^\#_{2n+1}$-like $\Pi^1_{2n+2}$-iterable $x$-premouse.  Let $\Sigma$ be a strategy for player $II$ in $\mathcal{G}(\M, 0, 2n+1)$. Let $\d$ be the least Woodin cardinal of $\M$ and suppose $\T$ is a correct iteration of $\M$ below $\d$ such that $\lh(\T)$ is a limit ordinal. For a limit $\a<\lh(\T)$ let $b_\a=[0, \a)_\T$.  Let $y_0$ be a real that codes the sequence $(\Q(b_\a, \T\rest \a): \a<\lh(\T)\wedge \a\in Lim)$ and $\M_{2n}^\#(\c(\T))$.  Finally set $b=\Sigma(\T, y)$ where $y$ is any real that is Turing above $y_0$ and also set $\N=\M_{2n}(\c(\T))$. Then the following holds.
\begin{enumerate}
\item $b$ is a cofinal branch. 
\item Suppose further that $\N\models ``\d(\T)$ is not a Woodin cardinal". Then $\Q(b, \T)$ exists and $\Q(b, \T)\insegeq \N$.
\item Suppose that $\N\models ``\d(\T)$ is a Woodin cardinal". Then $\N|(\d(\T)^+)^{\N}\insegeq \M^\T_b$.
\end{enumerate}
\end{proposition}
\begin{proof} 
Set $\a=\sup(b)$. We do the proof for the prototypical case $n=1$. 
 \begin{lemma} $b$ is cofinal.
 \end{lemma}
 \begin{proof} Towards a contradiction suppose that $\a<\lh(\T)$. \\\\
 \textbf{Case 1.}  $\Q(b, \T\rest \a)$ exists. \\\\
 Set $\R=\Q(b, \T\rest \a)$ and $\S=\Q(b_\a, \T\rest \a)$. Note that $\d(\T\rest \a)$ is a Woodin cardinal in both $\R$ and $\S$. We claim that $\R=\S$. To see this, we compare them inside $\M_1(\R, \S)$. If there is a successful coiteration of $\R$ and $\S$ then we indeed  have that $\R=\S$ (e.g. see \cite[Lemma 1.11]{PWIM}), and so any attempt to coiterate them is doomed to failure. 
 
 We attempt to coiterate $(\R, \S)$ in $\M_1(\R, \S)$ by building a coiteration $(\U, \W)$ in which the branches are picked as follows. Suppose $\l$ is a limit ordinal and we have defined $\U\rest \l$ and $\W\rest \l$. We want to describe our procedure for picking a branch of $\U\rest \l$ and $\W\rest \l$. As $\S$ is $\omega_1+1$-iterable inside $\M_1(\R, \S)$, we have a branch $d$ of $\W\rest \l$ that is according to the unique strategy of $\S$. Also, we must have that $\Q(d, \W\rest \l)$ exists as $\S$ is a $\d(\T\rest \a)$-mouse (in the sense of \cite[Definition 1.3]{PWIM})\footnote{Or just $\d(\T\rest \a)$-sound.}. We now seek a branch $c$ of $\U$ such that $\Q(c, \U\rest \l)$ exists and $\Q(c, \U\rest \l)=\Q(d, \W\rest \l)$. If there is such a branch then we choose it and continue the coiteration. Otherwise we stop the coiteration. As we cannot successfully coiterate $\R$ and $\S$, we must end up with a coiteration $(\U, \W)$ such that $\lh(\U)=\lh(\W)$ is a limit ordinal $< \omega_1^{\M_1(\R, \S)}$\footnote{One wrinkle here is to show that $\lh(\U)<\omega_1^{\M_1(\R, \S)}$. This follows from the fact that we always have a branch on the $\S$-side and so if the coiteration lasts $\omega_1^{\M_1(\R, \S)}$ steps then we could prove that $\U$ has a branch by a standard reflection argument.} and if $d$ is the branch of $\W$ according to $\S$'s unique strategy then there is no branch $c$ of $\U$  such that $\Q(c, \U)$ exists and $\Q(c, \U)=\Q(d, \W)$.
 
 We are now in the lucky situation that $\U$ is an allowed move for $I$ in the second round of $\mathcal{G}(\M, 0, 3)$. Indeed, some real recursive in $y$ codes $\S$, and therefore, $\U\in \M_1(y, \T, b)|\omega_1^{\M_1(y, \T, b)}$\footnote{Recall that for any $u\in \bR$ and any $\Pi^1_2$-iterable $\M_1$-like $\M$, $\M_1(u)|\omega_1^{\M(u)}\insegeq \M$. See \cite[Lemma 2.2]{PWIM}.}. Let then $I$ play $\U$ and $II$ play a branch $c$ such that $\M^\U_c$ is well-founded. Once again, we only know that $c$ is maximal. But now $\M^\U_c$ is $\Pi^1_2$-iterable above $\d(\U)$. Let $\gg=\sup(c)$, $\nu=\d(\U\rest \gg)$ and set $\R'=\Q(c, \U\rest \gg)$ if it exists and otherwise set $\R'=\M^{\U\rest \gg}_c$. If $\gg<\lh(\U)$ we set $\S'=\Q(c', \U\rest \gg)$ where $c'=[0, \gg)_{\U}$. Otherwise let $\S'=\Q(d, \W)$. In either of the cases, we have that $\S'$ is $\omega_1$-iterable, 1-small, $\d(\U)$-mouse and neither $\R'\not \insegeq \S'$ nor $\S'\not \insegeq \R'$. Moreover, $\nu$ is a Woodin cardinal in both $\R'$ and $\S'$. Thus, $\R'$ and $\S'$ have at least 2 Woodin cardinals, $\d(\T\rest \a)$ and $\nu$. 
 
 Because $\R'$ is $\Pi^1_2$-iterable above $\nu$ and $\S'$ is iterable, if $\R'=\Q(c, \U\rest \gg)$ then clause 3 of \cite[Lemma 2.2]{PWIM} implies that $\R'=\S'$,  which is a contradiction. Thus, $\R'=\M^{\U\rest \gg}_c$ and $\Q(c, \U\rest \gg)$ doesn't exist. It follows that $\pi^{\T^\frown \U}$ is defined, and hence, $\R'$ is not 1-small. But now \cite[Lemma 3.1]{PWIM} implies that $\S'\insegeq \R'$. Hence, after all, $\d(\U)$ is not a Woodin cardinal in $\R'$, contradiction. \\\\
 \textbf{Case 2.} $\Q(b, \T\rest \a)$ doesn't exist.\\\\
  Set now $\R=\M^{\T\rest \a}_b$ and $\S=\Q(b_\a, \T\rest \a)$. Let $\nu$ be the second least Woodin cardinal of $\R$. We now coiterate $\R_0=\R|(\nu^+)^\R$ and $\S$ inside $\M_1(\R_0, \S)$ using the same procedure as before, i.e., if $(\U_0, \W)$  is the coiteration that we build then at each limit stage $\l<\lh(\U_0)=\lh(\W)$, $d=[0, \l)_\W$ is the branch of $\W\rest \l$ according to the unique strategy of $\S$ as a $\d(\T\rest \a)$-mouse, and $c=[0,\l)_{\U_0}$ is the unique well-founded cofinal branch of $\U\rest \l$ such that $\Q(c, \U_0\rest \l)$ exists and $\Q(c, \U_0\rest \l)=\Q(d, \W\rest \l)$. 
 
Assume for a moment that the coiteration is successful. Thus $\U_0$ and $\W$ have last models, say $\R'_0$ and $\S'$. If $\S'\insegeq \R'_0$ then we in fact have that $\S\inseg \R_0$ and so $\d(\T\rest \a)$ is not a Woodin cardinal in $\R$. Thus, it must be the case that $\R'_0\insegeq \S'$. This means that $\pi^{\U_0}$ exists. Let now $\U$ be the copy of $\U_0$ on $\R$ using the identity map. The extenders of $\U$ and the tree structure of $\U$ are the same as the extenders of $\U_0$ and the tree structure of $\U_0$. Let $\R'$ be the last model of $\U$. 

Now $\U$ again is a valid move for $I$ in $\mathcal{G}(\M, 0, 3)$, and so we let $I$ play it. $II$ has to now either accept $\U$ or play a maximal branch of $\U$. If $II$ plays a maximal non-cofinal branch then we have two different branches $c_0, c_1$ of some $\U\rest \l$ such that both $\Q(c_0, \U\rest \l)$ and $\Q(c_1, \U\rest \l)$ exist, one of them is $\omega_1+1$ iterable and the other is $\Pi^1_2$-iterable. It then follows from clause 3 of \cite[Lemma 2.2]{PWIM} that  $\Q(c_0, \U\rest \l)=\Q(c_1, \U\rest \l)$ implying that $c_0=c_1$, contradiction. Thus, $II$ must accept it. But now we are in the second scenario of this repetitive argument, namely if $\nu'$ is the second Woodin cardinal of $\R'$ then $\R'$ as a $\nu'$-mouse is $\Pi^1_2$-iterable and not 1-small while $\S'$ as $\nu'$-mouse is $\omega_1+1$-iterable and is 1-small. Thus, \cite[Lemma 3.1]{PWIM} implies that  $\S'\insegeq \R'$, which in fact implies that $\S\insegeq \R$ and hence, $\d(\T\rest \a)$ is not a Woodin cardinal in $\R$.

We thus have that the coiteration $(\U_0, \W)$ of $(\R_0, \S)$ that we described above is not successful, which can only happen if we fail to find a branch $c$ of $\U_0$ such that $\Q(c, \U_0)$ exists and $\Q(c, \U_0)=\Q(d, \W)$, where once again $d$ is the unique branch of $\W$ that is according to the unique strategy of $\S$. We again let $\U$ be the copy of $\U_0$ onto $\R$ via identity and let $I$ play it in the second round of $\mathcal{G}(\M, 0, 3)$. Clause 3 of \cite[Lemma 2.2]{PWIM} implies that $II$ must play a cofinal branch $c$ of $\U$ such that $\M^\U_c$ is well-founded. Clause 3 of \cite[Lemma 2.2]{PWIM} implies that $\Q(c, \U)$ cannot exists while \cite[Lemma 3.1]{PWIM} implies that $\Q(d, \W)\insegeq \M^\U_c$.
\end{proof}
We thus have that $b$ is a cofinal branch. Assume then $\N\models ``\d(\T)$ is not a Woodin cardinal". Let $\S\insegeq \N$ be the longest such that $\S\models ``\d(\T)$ is a Woodin cardinal". Assume first that $\Q(b, \T)$ exists and set $\R=\Q(b, \T)$. We want to see that $\R=\S$, and this can be achieved by repeating the proof of Case 1 of the Lemma above. Assume then $\Q(b, \T)$ doesn't exist. We then let $\R=\M^\T_b$ and argue as in Case 2 of the Lemma above to conclude that $\S\insegeq \R$, contradicting the fact that $\Q(b, \T)$ doesn't exist.

Assume next that $\N\models ``\d(\T)$ is a Woodin cardinal". Let $\S\insegeq \N$ be such that $\S$ is a $\d(\T)$-mouse. We want to see that $\S\insegeq \M^\T_b$. Assume first that $\Q(b, \T)$ exists and set $\R=\Q(b, \T)$. We then compare $\R$ and $\S$ as in Case 1 of the Lemma and conclude that $\S\insegeq \R$ as $\R$ witnesses that $\d(\T)$ is not a Woodin cardinal while $\S$ does not. Suppose then $\Q(b, \T)$ doesn't exist and set $\R=\M^\T_b$. We then compare $\R$ with $\S$ as in Case 2 of the Lemma above. Just like in that proof, the conclusion once again is that $\S\insegeq \R$. 
\end{proof}

\section{The proof of \rthm{main theorem}}\label{the proof}

After setting up some notation, we will present the main ideas of the proof in \rsec{the main ideas}

Fix $n$. Below we prove \rthm{main theorem} for $n$. As \cite{Hjorth} proves \rthm{main theorem} for $n=0$, we might just as well assume that $n\geq 1$. As was already mentioned in clause 5 of \rrev{sargpre}, we have that
\begin{center}
$\sup_{m<\omega}\gg^{2n+1}_{m, \infty}=\k^1_{2n+3}$.
\end{center}
 Fix then $m<\omega$ such that $\gg^{2n+1}_{m, \infty}\geq \d^1_{2n+2}$. 
We then let  $(\S_0, \xi_0)$ be such that
\begin{enumerate}
\item $\S_0$ is a complete iterate of $\M^\#_{2n+1}$ such that $\T_{\M^\#_{2n+1}, \S_0}$ is below the least Woodin cardinal of $\M^\#_{2n+1}$,
\item $\xi_0\leq \gg_{m}^{\S_0}$,
\item $\pi_{\S_0, \infty}(\xi_0)=\d^1_{2n+2}$.
\end{enumerate}
Let $x_0$ be a real coding $(\S_0, \xi_0)$ and let ${\sf{Code}}$ be the set of reals $y$ such that $y$ codes a pair $(\R_y, \tau_y)$ with a property that 
\begin{enumerate}
\item $\R_y$ is a complete iterate of $\S_0$ such that $\T_{\S_0, \R_y}$ is below the least Woodin cardinal of $\S_0$,
\item $\tau_y<\pi_{\S_0, \R_y}(\xi_0)$. 
\end{enumerate}
For $x, y\in {\sf{Code}}$ set $x\leq^* y$ if and only if letting $\Q$ be the common complete iterate of $\R_x$ and $\R_y$, $\pi_{\R_x, \Q}(\tau_x)\leq \pi_{\R_y, \Q}(\tau_y)$. 
We have that there is a formula $\phi$ such that for all $(x, y)\in \bR^2$, 
\begin{center}
$x\leq^* y \iff \M_{2n}(x, y, x_0)\models \phi[x, y, s_m]$.
\end{center}
It follows that $\leq^*$ is $\Delta^1_{2n+3}(x_0)$ (see clause 10 of \rsec{pi iterability}). For $x\in \dom(\leq^*)$, let $\gg_x=\pi_{\R_x, \infty}(\tau_x)$. We have that the length of $\leq^*$ is $\d^1_{2n+2}$, and for each ordinal $\a<\d^1_{2n+2}$, there is $x\in \dom(\leq^*)$ such that $\gg_x=\a$.

Towards a contradiction, suppose $(A_\a: \a<\d^1_{2n+2})$ is a sequence consisting of distinct $\bS^1_{2n+2}$-sets.  
 Let $U\subseteq \bR^2$ be a universal $\Sigma^1_{2n+2}$ formula\footnote{I.e., for each $A\in \bS^1_{2n}$ there is $y\in \bR$ such that $A=\{u\in \bR: U(y, u)\}=_{def}U_y$ and moreover, if $y'\in \bR$ and $A\in \Sigma^1_{2n}(y')$ then there is a real $y$ recursive in $y'$ such that $A=U_y$. $U(y, u)$ essentially says that ``if $\phi$ is the $\Sigma^1_{2n}$-formula whose G\"odel number is $y(0)$ then $\phi[y', u]$ holds where $y'(n)=y(n+1)$".} and let for $\a<\d^1_{2n+2}$, $B_\a=\{ y: \{ z: U(y, z)\}=A_\a\}$.
Using the ${\sf{Coding\ Lemma}}$ (see \cite[Theorem 2.12]{Jackson}) we can find a real $z$ Turing above $x_0$ and a $\Sigma^1_{2n+3}(z)$ set $D^*\subseteq  \dom(\leq^*)\times \bR$ such that 
\begin{enumerate}
\item for each $x\in \dom(\leq^*)$, $D^*_x\not =\emptyset$,
\item for each $(x, y)\in D^*$, $y\in B_{\gg_x}$ (therefore, $U_y=A_{\gg_x}$).
\end{enumerate}
Let $D\in \Pi^1_{2n+2}(z)$ be such that $D\subseteq \bR^3$ and $(x, y)\in D^*\iff \exists u (x, y, u)\in D$. To get that $\dom(D^*)=\dom(\leq^*)$ we use the fact that $\Sigma^1_{2n+3}(z)$ is closed under $\exists^\bR$. For details, see the discussion above \cite[Lemma 2.13]{Jackson}.

We now set $\M=_{def}\M_{2n+1}(z)$. Given a complete iterate $\N$ of $\M$, let $\d_\N$ be the least Woodin cardinal of $\N$ and $\k_\N$ be the least $<\d_\N$-strong cardinal of $\N$. 

 If $x\in \bR$ codes a countable premouse then we let $\R_x$ be this premouse and let ${\sf{C}}'$ be the set of reals that code a countable $\Pi^1_{2n+2}$-iterable $\M_{2n+1}^\#$-like premouse over some real. For $x\in {\sf{C}}'$ we let $\nu_x$ be the least Woodin of $\R_x$ and $\mu_x$ be the least $<\nu_x$-strong of $\R_x$. Because $\R_x$ may not be iterable, $\R_x$ may not satisfy condensation (see \cite[Theorem 5.1]{OIMT}). Because for $u\in \bR$, $\M_{2n+1}^\#(u)$ is iterable it does have condensation, and in particular, it follows from \rlem{getting an fbcut} that $\M_{2n+1}(u)\models {\sf{Cond}}$. Let now ${\sf{C}}$ be the set of $x\in {\sf{C}}'$ such that 
\begin{enumerate}
\item $\R_x\models {\sf{Cond}}$,
\item $\R_x$ is translatable (see clause 3 of \rrev{s constructions}),
\item for all $\eta<\d$, $\R_x|\eta$ is $\eta+1$-iterable in $\R_x$ (see clause 11 of \rrev{s constructions}).
\end{enumerate}
 It follows from clause 1, 3 and 11 of \rrev{s constructions} that if $x$ codes $\M_{2n+1}^\#(u)$ then $x\in {\sf{C}}$.

 We now work towards applying Hjorth's reflection argument to produce a code of each $A_\gg$ below $\k_\M$. Let $w\subseteq \omega$ be a $\Pi^1_{2n+3}(z)$ real that is not $\Sigma^1_{2n+3}(z)$. Let $\psi$ be a $\Sigma^1_{2n+2}$ formula such that
\begin{center}
$n\in w\iff \forall z' \psi[n, z, z']$
\end{center}

\subsection{The main ideas of the proof}\label{the main ideas}

Before we go on, we give an outline of the proof that follows. The main idea of the proof is that each $A_\a=U_{w_0}$ for some $w_0$ that can be obtained ``below $\k_\M$". We say $w_0$ is a $U$ code for $A_\a$. This means that for each $\a<\d^1_{2n+2}$ we would like to find some complete iterate $\N$ of $\M$, $\xi<\k_\N$ and a $g\subseteq Coll(\omega, \xi)$ such that there is a $U$-code for $A_\a$ in $\N[g]$. 

Notice that, using genericity iterations, for each $\a<\d^1_{2n+2}$, we can find a complete iterate $\N$ of $\M$ such that $A_\a$ has a $U$-code that is $\N$-generic for the extender algebra at $\d_\N$. Because $\N\models {\sf{Cond}}$, this allows us to obtain a $\xi<\k_\N$ such that $A_\a$ has a $U$-code that is generic over the extender algebra of $\N|\xi$. To make this work, we need to make sure that the $U$-code of $A_\a$ can be identified in a  first order way over $\N$, and this is achieved by the formula $\theta_0$ introduced below.

However, the above is not enough. We need to find $\N$ as above which is $\a$-stable in the sense that for any complete iterate $\Q$ of $\N$ there is such a $w_0$ in $\Q^{Coll(\omega, \pi_{\N, \Q}(\xi))}$. Such a uniformity is not easy to achieve, and here, we use our reals $w$ and $w'$ introduced below. $w$ is a $\Pi^1_{2n+3}(z)$ real which is not $\Sigma^1_{2n+3}(z)$ while $w'$ is a $\Sigma^1_{2n+3}(z)$ subset of $w$. 

The $\Pi^1_{2n+2}(z)$-formula $\theta$ is the key. $w'$ is the real defined by $\exists \vec{y}\theta(x, \vec{y})$. The formulas are defined in a way that assures the following. If $k\in w'$ then the least initial segment of $\M$ that ``proves" that $k\in w$ is below the least initial segment of $\M$ that has a $U$-code for $A_\a$. Thus, if $k\in w-w'$ then for any $\a$ (roughly speaking), the least initial segment of $\M$ that has a $U$-code for $A_\a$ is below the least initial segment that proves that $k\in w$. This means that if we pick $\zeta$ to be the least such that $\M|\zeta$ proves that $k\in w$ then each $A_\a$ will have a $U$-code below $\zeta$. 

Above by ``the least initial segment of $\M$ that proves $k\in w$" we mean that the formula $``k\in w"$ is first order over $\M$ as explained in \rprop{correctness pi}, and therefore, because $\M\models {\sf{Cond}}$, we have some $\xi<\k_\M$ such that $``k\in w"$ is first order over $\M|\xi$. The least such $\xi$, which we call $\zeta$, determines the least level of $\M$ that proves that $k\in w$.

Because of ${\sf{Cond}}$, we can assume that $\zeta$ is a cutpoint of $\M$. We then finish the proof by showing that if $\Q=\M|\zeta$ then (roughly) each $A_\a$ has a $U$-code in a generic extension of a complete iterate of $\Q$. But the set of complete iterates of $\Q$ is too simple (this follows from calculations presented in \rlem{complexity}), and one could not define $\d^1_{2n+2}$ many distinct $\mathbf{\Sigma}^1_{2n+2}$ sets whose $U$-codes can be obtained in generic extensions of complete iterates of $\Q$. This then gives a contradiction as $(A_\a: \a<\d^1_{2n+2})$ consists of distinct $\mathbf{\Sigma}^1_{2n+2}$ sets whose $U$-codes can be obtained in generic extensions of complete iterates of $\Q$.
 
\subsection{The formula $\theta$}\label{sec:theformtheta}

We will need the following lemma. Recall that $U$ is our fixed universal $\Sigma^1_{2n+2}$ set.

\begin{lemma}\label{existence of conditions} Suppose $\N$ is a $\M_{2n+1}^\#$-like $\Pi^1_{2n+2}$-iterable $z$-premouse, $(\nu_0, \nu_1,..., \nu_{2n})$ are the Woodin cardinals of $\N$ enumerated in increasing order, $\mathcal{E}$ is a weakly appropriate set of extenders (relative to $\nu_0$), $\nu<\nu_0$ and suppose $(w_0,( w_1, w_2, w_3))$ is a sequence of reals such that 
\begin{enumerate}
\item $D(w_1, w_2, w_3)$ and $U(w_2, w_0)$,
\item $w_0$ and $(w_1, w_2, w_3)$ are generic over $\N$ for ${\sf{Ea}}^\N_{\nu_0, \nu, \mathcal{E}}$. 
\end{enumerate}
There is then a condition $(p, q)\in {\sf{Ea}}^\N_{\nu_0, \nu, \mathcal{E}}\times {\sf{Ea}}^\N_{\nu_0, \nu, \mathcal{E}}$ such that 
\begin{enumerate}
\item $w_0\models p$ and $(w_1, w_2, w_3)\models q$,
\item $\N\models (p, q)\forces ``D({\sf{ea}}^r_1, {\sf{ea}}^r_2, {\sf{ea}}^r_3) \wedge U( {\sf{ea}}^r_2, {\sf{ea}}^l)"$.
\end{enumerate}
\end{lemma}
\begin{proof} First of all, notice that $\N=\M_{2n}^\#(\N|\nu_0)$ (see clause 11 of \rsec{pi iterability}). Suppose $w_1$ codes a pair $(\R, \xi)$. Working in $\N[w_1, w_2, w_3]$, let $\gg=\pi_{\R, \mH^\N}(\xi)$. We then have that there is a condition $q_0\in {\sf{Ea}}^\N_{\nu_0, \nu, \mathcal{E}}$ such that $(w_1, w_2, w_3)\models q_0$ and  $q_0\forces ``$if $(\S, \b)$ is the pair coded by $\sf{ea}_1$ then $\pi_{\S, \mH}(\b)=\check{\gg}$"\footnote{Here we use the fact that the directed system can be internalized to $\N$. See clause 8 of \rrev{gamma stable}.} and also, $q_0\forces D({\sf{ea}}_1, {\sf{ea}}_2, {\sf{ea}}_3)$. Let now $(g_1, g_2, g_3)$ be $\N[w_0]$-generic for ${\sf{Ea}}^\N_{\nu_0, \nu, \mathcal{E}}$ such that $(g_1, g_2, g_3)\models q_0$. Because $\N=\M_{2n}^\#(\N|\nu_0)$ (implying that $\N$ is $\Sigma^1_{2n+2}$-correct), it follows that $g_1\in {\sf{Code}}$ and that $w_0\in A_{\gg_{g_1}}$,\footnote{Notice that because $w_1\in {\sf{Code}}$ we really do have that $\R$ is a complete iterate of $\M_{2n+1}^\#$. Because $\mH^\N$ is a complete iterate of $\R$, this implies $\pi_{\mH^\N, \infty}(\gg)=\gg_{w_1}$. Because $D(g_1, g_2, g_3)$ holds (by absoluteness, see \rrev{ea}), we have that if $(\R', \xi')$ is the pair coded by $g_1$ then $\R'$ is a complete iterate of $\M_{2n+1}^\#$ and because $\pi_{\R', \mH^\N}(\xi')=\gg$, we have that $\gg_{g_1}=\gg_{w_1}$. Thus, $w_0\in A_{\gg_{g_1}}$.},\footnote{The fact that $\mH^\N$ is a complete iterate of $\R$ follows from generic comparisons (this point also appeared in \rlem{stable p}). Recall that $\mH^\N$ is the direct limit of all complete iterates of $\M_{2n+1}^\#$ that are in $\N|\nu$ where $\nu$ is the least inaccessible cardinal of $\N$ above the least Woodin cardinal of $\N$. Just like in the case of HOD analysis of $L[x][G]$ presented in \cite{HODCoreModel}, this direct limit can be shown to coincide with the direct limit of all iterates of $\M_{2n+1}^\#$ that appear in $N[H]$ where $H\subseteq Coll(\omega, <\nu)$ generic over $\N$. These points are explained in detail in various publications such as  \cite{SargPre}, \cite{Sandra}, \cite{Varsovian} and \cite{HODCoreModel}.}. Then we must have that $\N[w_0][(g_1, g_2, g_3)]\models U(g_2, w_0)$. The desired $(p, q)$ can now be found using the forcing theorem.
\end{proof}
\textbf{The formula $\theta_0$:}\\\\
Let $\theta_0(u, w_0, w_1, w_2, w_3)$ be the  $\Pi^{1}_{2n+2}(z)$-formula that is the conjunction of the following clauses:
\begin{enumerate}
\item $D(w_1, w_2, w_3)$\footnote{Notice that $z$ is implicitly build into $D$.},
\item $u\in {\sf{C}}$ and $\R_u$ is $\Pi^1_{2n+2}$-iterable,
\item $(z, w_0, w_1, w_2, w_3)\in \R_u\cap \bR^5$,
\item $R_u\models U(w_2, w_0)$.
\end{enumerate}

At this point the reader may find it useful to review \rprop{correctness1}, clause 3 and 6 of \rrev{ea}, and clause 2 of \rrev{s constructions}. 

Let $U'$ be a $\Pi^1_{2n+1}$ set such that $U(\vec{a})\iff \exists b U'(b,\vec{a})$. 

Suppose now that $\theta_0(u, w_0, w_1, w_2, w_3)$ holds. Let $\K={\sf{StrLe}}(\R_u, z)$ and $\mathcal{E}=\mathcal{E}_{sm}^\K$ (see  \rdef{appropriate sequence}). It follows from \rlem{existence of conditions} that there is a condition $(p, q)\in {\sf{Ea}}^\K_{\nu_u, \mathcal{E}}\times {\sf{Ea}}^\K_{\nu_u, \mathcal{E}}$ such that
\begin{enumerate} 
\item $w_0\models p$,
\item $(w_1, w_2, w_3)\models q$, and
\item $(p, q)\forces \exists b( U'_{\K[({\sf{ea}}^l, {\sf{ea}}^r)], d}(b, {\sf{ea}}^r_2, {\sf{ea}}^l))$ where letting $(\nu_u, \xi_1, ..., \xi_{2n})$ be the Woodin cardinals of $\K$ enumerated in increasing order, $d=(\xi_1, \xi_2,..., \xi_{2n})$.\footnote{See item 3 of \ref{ea}. According to the terminology defined there, we have that $U'_{\K[({\sf{ea}}^l, {\sf{ea}}^r)], d}(b, {\sf{ea}}^r_2, {\sf{ea}}^l)$ is an instances of $\phi_{\M, d}$ with $\phi$ the $\Pi^1_{2n+1}$ formula defining $U'$ and $\M=\K[({\sf{ea}}^l, {\sf{ea}}^r)]$.}
\end{enumerate}
Clause 3 above is a consequence of \rprop{correctness1} and the fact that because $\R_u$ is $\Pi^1_{2n+2}$-iterable, it is $\omega_1+1$-iterable above $\nu_u$ (see clause 11 of \rsec{pi iterability} and recall that under $\sf{AD}$, $\omega_1$-iterability implies $\omega_1+1$-iterability). Notice that to verify $U'$ we skip the extender algebra ${\sf{Ea}}^\K_{\xi_1, \nu_u}$. However, $\xi_1$ is not a lazy Woodin cardinal, it has a noble role of making sure that the witness $b$ can be found inside $\K[({\sf{ea}}^l, {\sf{ea}}^r)]$. The fact that we have skipped $\xi_1$ will be used in the proof of \rlem{w is subset}.
 
Because $\R_u\models {\sf{Cond}}$, we have that if $\theta_0(u, w_0, w_1, w_2, w_3)$ holds then in fact there are $\a<\b<\nu_u$ such that
\begin{enumerate}
\item $\R_u|\b\models {\sf{ZFC}}+``$there are $2n+1$ many Woodin cardinals",
\item $\R_u|\b\models ``\a$ is the least Woodin cardinal",
\item $\a$ is an $fb$-cut of $\R_u$ (see \rdef{fbcut}, in particular, $\a$ is an $\R_u$-cardinal),
\item setting $\K={\sf{StrLe}}(\R_u|\b, z)$ and $\mathcal{E}=\mathcal{E}^\K_{sm}$, there is a condition $(p, q)\in {\sf{Ea}}^\K_{\a, \mathcal{E}}\times {\sf{Ea}}^\K_{\a, \mathcal{E}}$ such that 
\begin{enumerate}
\item $w_0\models p$,
\item $(w_1, w_2, w_3)\models q$, and
\item $(p, q)\forces \exists b(U'_{\K[({\sf{ea}}^l, {\sf{ea}}^r)], d}({\sf{ea}}^r_2, {\sf{ea}}^l))$  where letting $(\a, \xi_1, ..., \xi_{2n})$ be the Woodin cardinals of $\K$ enumerated in increasing order, $d=(\xi_1, \xi_2,..., \xi_{2n})$.
\end{enumerate}
\end{enumerate}
If $s=(u, w_0, w_1, w_2, w_3)$ then we say that $(\a, \b)$ witnesses $\theta_0$-reflection for $s$ if $(\a, \b)$ satisfies clause 1-4 above. Assuming $\theta_0(u, w_0, w_1, w_2, w_3)$ holds with $s=(u, w_0, w_1, w_2, w_3)$, we let $(\a_s, \b_s)$ be the lexicographically least witnessing $\theta_0$-reflection for $s$. Because $\R_u\models {\sf{Cond}}$ we have that  $(\a_s, \b_s)\in \mu_u^2$. Notice that $(\a_s, \b_s)$ is definable from $(z, w_0, w_1, w_2, w_3)$ over $\R_u$ in the sense that there is a formula $\theta'$ such that $(\a_s, \b_s)$ is the unique pair $(\a, \b)$ such that \begin{center}
$\R_u\models \theta'[(\a, \b), z, w_0, w_1, w_2, w_3]$.
\end{center}
\textbf{The formula $\theta$:}\\\\
Let $\psi'$ be a $\Pi^1_{2n+1}$ such that $\psi(...)\iff \exists v\psi'(..., v)$ ($\psi$ is the formula used to define $w$, see just before \rsec{sec:theformtheta}).
Let $\theta(k, u, w_0, w_1, w_2, w_3)$ be the  $\Pi^1_{2n+2}(z)$-formula that is the conjunction of the following clauses:
\begin{enumerate}
\item $k\in \omega$,
\item $\theta_0(u, w_0, w_1, w_2, w_3)$,
\item letting $s=(u, w_0, w_1, w_2, w_3)$, in $\R_u$,\footnote{Here we need to add ``if $(\a_s, \b_s)$ is such that $\R_u\models \theta'[(\a, \b),z, w_0, w_1, w_2, w_3]$ then".} there are ordinals $\a'<\b'<\a_s$ such that 
\begin{enumerate}
\item $\R_u|\b'\models ``{\sf{ZFC}}+``$there are $2n+1$ many Woodin cardinals" +$``\a'$ is the least Woodin cardinal",
\item $\a'$ is an $fb$-cut in $\R_u|\a_s$ (and hence, in $\R_u$), and\\\\
setting $\K={\sf{StrLe}}(\R_u|\b')$, $\mathcal{E}=\mathcal{E}_{sm}^{\K|\a'}$ and for $i\in [1, 2n]$, letting $\xi_i$ be the $i+1$st Woodin cardinal of $\K$ and $d=(\xi_1, \xi_2, ..., \xi_{2n})$ the following holds:
\item If $p\in {\sf{Ea}}^{\K}_{\a', \mathcal{E}}$ then $\K\models p\forces \exists v(\psi'_{\K[{\sf{ea}}], d}[k, z, {\sf{ea}}, v])$.
\end{enumerate}
\end{enumerate}
\subsection{The real $w'$}
Set
\begin{center}
 $k\in w' \iff \exists u, w_0, w_1, w_2, w_3 (\theta[k,u, z, w_0, w_1, w_2, w_3])$.
 \end{center}
 
Clearly $w'$ is $\Sigma^1_{2n+3}(z)$.  The proof of the next lemma is similar to the proof of Claim 2 that appears in the proof of \cite[Theorem 3.3]{Hjorth}.
 
 \begin{lemma}\label{w is subset} $w'\subseteq w$.
 \end{lemma}
 \begin{proof} 
 Suppose $k\in w'$. Let $\Sigma$ be a winning strategy for $II$ in $\mathcal{G}(\R_u, 0, 2n+1)$.  We want to argue that $k\in w$. Fix $s=(u, w_0, w_1, w_2, w_3)$ such that $\theta[k,u, z, w_0, w_1, w_2, w_3]$ holds. We want to show that for each $z'$, $\psi[k, z, z']$ holds. Fix then such a $z'$. 
 
  We have that $\R_u$ is a properly $2n+1$-small $\Pi^1_{2n+2}$-iterable premouse such that $(z, w_0, w_1, w_2, w_3)\in \R_u\cap \bR^5$. We now want to use $\Pi^1_{2n+2}$-iterability of $\R_u$ to make $z'$ generic over the image of ${\sf{Ea}}^{\R_u}_{\nu_u, \mathcal{E}^{\R_u}_{le}}$\footnote{Recall our notation: $\nu_u$ is the least Woodin cardinal of $\R_u$.}. We iterate $\R_u$ producing iteration $\T$ of $\R_u$ by using the rules of $(z', \nu_u, 0)$-genericity iteration to pick extenders (see clause 2 of \rrev{ea}) and by picking branches so that the iteration stays correct. More precisely, if $\gg<\lh(\T)$ and $c=[0, \gg)_{\T}$ then $\Q(c, \T\rest \gg)$ exists and $\Q(c, \T\rest \gg)\insegeq \M_{2n}(\c(\T\rest \gg))$. Thus the resulting iteration $\T$ of $\R_u$ is correct, has no drops and is below $\nu_u$\footnote{Recall our setup, $\nu_u$ is the least Woodin cardinal of $\R_u$.}. These are the two possibilities. Either\\\\
(Pos1) $\T$ has a last model $\R'$ and $z'$ is generic over $\R'$ for ${\sf{Ea}}^{\R'}_{\mu, \mathcal{E}^{\R'}_{le}}$ where $\mu=\pi^{\T}(\nu_u)$, or\\
(Pos2) $\T$ does not have a last model and $\lh(\T)$ is a limit ordinal. \\
\begin{sublemma}\label{cut back} There is $\xi<\lh(\T)$ and $\a'<\b'$ such that the following conditions hold:
\begin{enumerate}
\item  The generators of $\T\rest \xi$ are contained in $\a'$.
\item  $\M^\T_\xi|\b'\models {\sf{ZFC}}+``$there are $2n+1$ many Woodin cardinals" +$``\a'$ is the least Woodin cardinal".
\item Setting $\K={\sf{StrLe}}(\M^\T_\xi|\b')$, $\mathcal{E}=\mathcal{E}_{sm}^{\K|\a'}$ and for $i\in [1, 2n]$, letting $\xi_i$ be the $i+1$st Woodin cardinal of $\K$ and $d=(\xi_1, \xi_2, ..., \xi_{2n})$ the following holds:
\begin{enumerate}
\item If $p\in {\sf{Ea}}^{\K}_{\a', \mathcal{E}}$ then $\K\models p\forces \exists v(\psi'_{\R_u[{\sf{ea}}], d}[k, z, {\sf{ea}}, v])$.
\item $z'$ satisfies all the axioms of ${\sf{Ea}}^{\K}_{\a', \mathcal{E}}$.
\item If player $I$ plays $(\T\rest \xi+1, y)$ where $y\in \bR$ is any real coding $\T\rest \xi+1$ then $\Sigma(\T\rest \xi+1, y)={\sf{accept}}$ and the iterations of $\K[z']$ that are above $\a'$ are legal moves for player $I$ in the second round of $\mathcal{G}(\R_u, 0, 2n+1)$.
\end{enumerate}

\end{enumerate}
\end{sublemma}
\begin{proof}\\\\
\textbf{Suppose first that (Pos1) holds.} Set $(\a, \b)=\pi^\T(\a_s, \b_s)$. We have that $\a<\b<\mu$ and $\a$ is an $fb$-cut in $\R'$. Because $k\in w'$ we have $\a'<\b'<\a$ such that\\\\
(A1) $\R'|\b'\models {\sf{ZFC}}+``$there are $2n+1$ many Woodin cardinals" +$``\a'$ is the least Woodin cardinal",\\
(B1) $\a'$ is an $fb$-cut in $\R'$, and\\\\
setting $\K={\sf{StrLe}}(\R'|\b')$, $\mathcal{E}=\mathcal{E}_{sm}^{\K|\a'}$ and for $i\in [1, 2n]$, letting $\xi_i$ be the $i+1$st Woodin cardinal of $\K$ and $d=(\xi_1, \xi_2, ..., \xi_{2n})$ the following holds:\\\\
(C1) If $p\in {\sf{Ea}}^{\K}_{\a', \mathcal{E}}$ then $\K\models p\forces \exists v(\psi'_{\R_u[{\sf{ea}}], d}[k, z, {\sf{ea}}, v])$.\\\\
We also have that if $\mathcal{E}'=\mathcal{E}^{\R'|\a}_{le}$ then $z'$ satisfies all the axioms of ${\sf{Ea}}^{\R'|\a}_{\a,\mathcal{E}'}$, and hence,\\\\
(D1) $z'$ satisfies all the axioms of ${\sf{Ea}}^{\K}_{\a', \mathcal{E}}$. \\\\
To see (D1), we use \rlem{generic over le} and the fact that both $\a$ and $\a'$ are $fb$-cuts of $\R'$.

 Let $\xi<\lh(\T)$ be the least such that $\M_\xi^\T|\a'=\R'|\a'$. It follows that\\\\
 (E1) the generators of $\T\rest \xi$ are contained in $\a'$ and\\
 (F1) if $y$ is a real coding $\T\rest \xi+1$ then $\Sigma(\T\rest \xi+1, y)={\sf{accept}}$ (see  \rprop{correct branch}). \\\\
 (E1) and (F1) then imply the following.\\\\
(G1) If player $I$ plays $(\T\rest \xi+1, y)$ where $y\in \bR$ is any real coding $\T\rest \xi+1$, $\Sigma(\T\rest \xi+1, y)={\sf{accept}}$ and iterations of $\K[z']$ that are above $\a'$ are legal moves for player $I$ in the second round of $\mathcal{G}(\R_u, 0, 2n+1)$.\\\\
Because $\M_\xi^\T|\b'=\R'|\b'$ we have that (A1)-(G1) imply clauses (1)-(3) of \rsublem{cut back}. Indeed, (1) is a consequence of (E1), (2) is a consequence of (A1), (3a) is a consequence of (C1), (3b) is a consequence of (D1) and (3c) is a consequence of (G1). \\\\
\textbf{Suppose next that (Pos2) holds.} Set $\N=\M_{2n}(\c(\T))$ and $\k=\d(\T)$. \\\\
\textit{Claim.} Let $y\in \bR$ code the pair $(\T, \N|(\k^+)^\N)$ and set $b=\Sigma(\T, y)$. Then
\begin{enumerate}
\item $\N\models ``\k$ is a Woodin cardinal" and $\N|(\k^+)^\N\insegeq \M^\T_b$ and
\item $\k$ is an $fb$-cut of $\M^\T_b$.
\end{enumerate}
\begin{proof} Clause 1 is a consequence of \rprop{correct branch}. To see clause 2 let $\X$ be the fully backgrounded construction of $\c(\T)=\M^\T_b|\k$. Suppose there is $\Y$ which extends $\X$, is constructed by the fully backgrounded construction of $\M^\T_b|\pi^\T_b(\nu_u)$ and $\rho_\omega(\Y)<\k$. Let $F$ be an extender used in $b$ such that $\cp(F)>\rho_\omega(\Y)$\footnote{It might help to review \rcor{cor to above}}. Let $\xi$ be such that $E_\xi^\T=F$. Then $\cp(F)$ is an $fb$-cut of $\M^\T_{\xi+1}$ and $\cp(\pi_{\xi+1, b}^\T)>\cp(F)$ implying that in fact $\cp(F)$ is an $fb$-cut in $\M^\T_b$. Hence, $\rho_\omega(\Y)\geq \cp(F)$, contradiction. 
\end{proof}

Using condensation (applied in $\N$) and \rlem{existence of conditions}, we can find $\b\in (\k, (\k^+)^\N)$ such that
\begin{enumerate}
\item $\N|\b\models {\sf{ZFC}}+``$there are $2n+1$ many Woodin cardinals",
\item setting $\K={\sf{StrLe}}(\N|\b, z)$ and $\mathcal{E}=\mathcal{E}_{sm}^\K$, there is a condition $(p, q)\in {\sf{Ea}}^\K_{\k, \mathcal{E}}\times {\sf{Ea}}^\K_{\k, \mathcal{E}}$ such that 
\begin{enumerate}
\item $w_0\models p$,
\item $(w_1, w_2, w_3)\models q$, and
\item $(p, q)\forces \exists b(U'_{\K[({\sf{ea}}^l, {\sf{ea}}^r)], d}({\sf{ea}}^r_2, {\sf{ea}}^l))$ where letting $(\k, \xi_1, ..., \xi_{2n})$ be the Woodin cardinals of $\K$ enumerated in increasing order, $d=(\xi_1, \xi_2,..., \xi_{2n})$.
\end{enumerate}
\end{enumerate}
Let now $y_0\in \bR$ code the pair $(\T, \N|(\k^+)^\N)$ and set $b=\Sigma(\T, y_0)$. We then have that $\N|\b\insegeq \M^\T_b$ (see clause 1 of the Claim). Notice that because $\k$ is an $fb$-cut of $\M^\T_b$ (see clause 2 of the Claim) and because $\N|\b\insegeq \M^\T_b$, the above clauses imply that $\pi^\T_b(\a_s)\leq \k$.  Because $k\in w'$, we must have $\a'<\b'<\pi^\T_b(\a_s)\leq \k$ such that\\\\
(A2) $\M^\T_b|\b'\models ``{\sf{ZFC}}+``$there are $2n+1$ many Woodin cardinals" +$``\a'$ is the least Woodin cardinal",\\
(B2) $\a'$ is an $fb$-cut in $\M^\T_b$, and\\\\
setting $\K={\sf{StrLe}}(\M^\T_b|\b')$, $\mathcal{E}=\mathcal{E}_{sm}^{\K|\a'}$ and for $i\in [1, 2n]$, letting $\xi_i$ be the $i+1$st Woodin cardinal of $\K$ and $d=(\xi_1, \xi_2, ..., \xi_{2n})$ the following holds:\\\\
(C2) If $p\in {\sf{Ea}}^{\K}_{\a', \mathcal{E}}$ then $\K\models p\forces \exists v(\psi'_{\R_u[{\sf{ea}}], d}[k, z, {\sf{ea}}, v])$.\\\\
Let then $\K={\sf{StrLe}}(\M^\T_b|\b')$. Because \\\\
(a) $z'$ satisfies all the axioms of ${\sf{Ea}}^{\c(\T)}_{\d(\T), \mathcal{E}'}$ where $\mathcal{E'}=\mathcal{E}^{\c(\T)}_{le}$, and \\ 
(b) $\a'$ is an $fb$-cut of $\M^\T_b$\\\\
(D2) $z'$ satisfies all the axioms of ${\sf{Ea}}^{\K}_{\a', \mathcal{E}}$. \\

Let $\xi<\lh(\T)$ be the least such that $\M_\xi^\T|\a'=\M^\T_b|\a'$ (notice that because $\a'<\d(\T)$, we really do have such a $\xi<\lh(\T)$). It follows that\\\\
(E2) the generators of $\T\rest \xi$ are contained in $\a'$.\\\\
(F2) if $y$ is a real coding $\T\rest \xi+1$ then $\Sigma(\T\rest \xi+1, y)={\sf{accept}}$ (see  \rprop{correct branch}). \\\\
 (E2) and (F2) then imply the following.\\\\
(G2) If player $I$ plays $(\T\rest \xi+1, y)$ where $y\in \bR$ is any real coding $\T\rest \xi+1$, $\Sigma(\T\rest \xi+1, y)={\sf{accept}}$ and iterations of $\K[z']$ that are above $\a'$ are legal moves for player $I$ in the second round of $\mathcal{G}(\R_u, 0, 2n+1)$.\\\\
As in the case of (Pos1), (A2)-(G2) imply that (1)-(3) of \rsublem{cut back} hold.
\end{proof}
Let then $\xi$, $d=(\xi_1, \xi_2, ..., \xi_{2n})$ and $\K$ be as in the conclusion of \rsublem{cut back}. We now want to conclude that $\psi[k, z, z']$ holds. Let $v\in \K[z']$ be such that $\K[z']\models \psi'_{\K[z'], d}[k, z, z', v]$. It is then enough to show that $\psi'[k, z, z', v]$ holds, and this will be established by \rsublem{psi holds}.\\
 
 \begin{sublemma}\label{psi holds}
 $\psi'[k, z, z', v]$ holds.
 \end{sublemma}
 \begin{proof}
 $\psi'[k, z, z', v]$ is a $\Pi^1_{2n+1}$-formula, so we can find a $\Sigma^1_{2n}$ formula $\psi''$ such that $\psi'[k, z, z', v]\iff \forall t\psi''[k, z, z', v, t]$. Fix $t\in \bR$. We now want to argue that $\psi''[k, z, z', v, t]$ holds. Let $I$'s first move in  $\mathcal{G}(\R, 0, 2n+1)$ be $(\T\rest \xi+1, t')$ where $t'$ is a real coding $t$ and $\T\rest \xi+1$. As discussed above, $II$ must accept $\T\rest \xi+1$.  
 
 Notice now that $\K$ as a $\K|\xi_1$-mouse is $2n-1$-small. Working inside $\M_{2n-1}(t')$, let $\U$ be an iteration of $\K$ such that 
 \begin{enumerate}
 \item $\U$ is a $(t, \xi_2, \xi_1)$-genericity iteration,
 \item for each limit ordinal $\l<\lh(\U)-1$, if $c=[0, \l)_{\U}$ then $\Q(c, \U)$ exists and is $\Q(c, \U\rest \l)\insegeq \M_{2n-2}(\c(\U\rest \l))$.
 \end{enumerate}
As always we have two possibilities: either\\\\
(1) $\U$ has a last model $\K_1$ such that $t$ is generic over $\K_1$ for ${\sf{Ea}}^{\K_1}_{\pi^{\U}(\xi_2), \xi_1}$ or\\
(2) $\U$ is of limit length and there is no branch $c$ of $\U$ such that  $\Q(c, \U)\insegeq \M_{2n-2}(\c(\U))$\footnote{Here we should say that ``there is no branch $c\in \M_{2n-1}(t')$ such that..." but this is not necessary as if there was such a branch then it had to be in $\M_{2n-1}(t')$.}.\\\\
We now let $I$ play $\U$ in the second round of $\mathcal{G}(\R, 0, 2n+1)$. If $II$ accepts it then let $\K_1$ be the last model of $\U$. If $II$ plays a maximal well-founded branch $c$ then let $\K_1=\M^{\U\rest \sup(c)}_c$\footnote{While we do not need this, it can be argued, using the minimality of $\K$, that $II$ cannot play a branch.}. 

Suppose for a moment that $II$ plays a branch $c$. Set $\iota=\sup(c)$. We want to argue that $\Q(c, \U\rest \iota)$ doesn't exists. If it does then it is $2n-2$-small $\d(\U\rest \iota)$-mouse which is $\Pi^1_{2n}$-iterable above $\d(\U\rest \iota)$. It follows from clause 11 of \rrev{pi iterability} that $\Q(c, \U\rest \iota)$ is iterable and hence, $\Q(c, \U\rest \iota)\insegeq \M_{2n-2}(\c(\U\rest \iota))$, which is a contradiction. 

Let now $j:\K\rightarrow \K_1$ be the iteration map given by $\pi^{\U}$ or $\pi^{\U\rest \iota}_{c}$ depending on which case was used to define $\K_1$. In both of these cases we have that $t$ is generic over $\K_1$ for ${\sf{Ea}}^{\K_1}_{j(\xi_2), \xi_1}$. Let $d'=(j(\xi_2), j(\xi_3), ..., j(\xi_{2n}))$. We have that $\K_1[z', t]\models \psi''_{\K_1[z', t], d'}[k, z, z', v, t]$. Moreover, $\K_1$ above $j(\xi_2)$ is $2n-2$-small and $\Pi^1_{2n}$-iterable. It follows from clause 11 of \rrev{pi iterability} that $\K_1$ is iterable above $j(\xi_2)$ and hence, applying \rprop{correctness} to $\K_1[z', t]$, we get that $\psi''[k, z, z', v, t]$ holds. 
  \end{proof}
 \end{proof}
 
 \subsection{Hjorth's reflection}\label{sec:href}
 
Recall that $w$ is $\Pi^1_{2n+3}(z)$ but not $\Sigma^1_{2n+3}(z)$, and since $w'$ is $\Sigma^1_{2n+3}(z)$, we have that $w-w'$ is not empty. Fix now $k\in w-w'$ and  let  $\a<\b<\k_\M$ be the least such that 
 \begin{enumerate}
 \item $\M|\b\models {\sf{ZFC}}$+``there are $2n+1$ Woodin cardinals",
 \item $\M|\b\models ``\a$ is a Woodin cardinal",
 \item $\a$ is both an inaccessible $fb$-cut of  $\M$ and a cutpoint of $\M$,
 \item letting $(\a, \tau_1, ..., \tau_{2n})$ be the Woodin cardinals of $\M|\b$ enumerated in the increasing order, $d=(\tau_1, ..., \tau_{2n})$ and $\K=\sf{StrLe}(\M|\b)$, whenever $q\in {\sf{Ea}}^{\K}_{\a, sm}$, $\K\models q\forces \exists v\psi'_{\M|\b, d}[k, z, {\sf{ea}}, v]$.
 \end{enumerate}
 Notice that to get $\b<\k_\M$ we are using that $\M\models {\sf{Cond}}$. Also, notice that if $\N$ is a complete iterate of $\M$ then setting $\pi_{\M, \N}(\a, \b)=(\iota_\N, \zeta_\N)$, we have that $(\iota_\N, \zeta_\N)$ has the same definition over $\N$ as $(\a, \b)$ over $\M$. 
Set $\Q=\M|(\zeta_\M^+)^\M$\footnote{Notice that $(\iota_\M^+)^\M=(\zeta_\M^+)^\M$.}. Given a complete iterate $\N$ of $\M$ we let $\Q_\N=\pi_{\M, \N}(\Q)$. Recall that we set $B_\gg=\{ y: U_y=A_\gg\}$. To implement Hjorth's reflection we will need the following general lemma which is based on \cite{PWIM}.

\begin{lemma}\label{complexity} Suppose $\R$ is a complete iterate of $\M_{2n+1}$ such that $\lh(\T_{\M_{2n+1}, \R})<\omega_1$. Let $\a\leq \k_\R$ be a cutpoint of $\R$ and set $\S=\R|(\a^+)^\R$. Let $x\in \bR$ code $\S$. Then the following statements hold:
\begin{enumerate}
\item The statement that $``y$ codes a complete iterate of $S$" is $\Sigma^1_{2n+2}(x)$.
\item The statement that $``y$ codes a pair $(\N, \xi)$, $y'$ codes a pair $(\N', \xi')$, $\N$ and $\N'$ are complete iterates of $\S$ and if $\N''$ is the common complete iterate of $\N$ and $\N'$ then $\pi_{\N, \N''}(\xi)\leq \pi_{\N', \N''}(\xi')$" is $\Sigma^1_{2n+2}(x)$. 
\end{enumerate}
\end{lemma}
\begin{proof}
To say that $``y$ codes a complete iterate of $\S$" it is enough to say the following:
\begin{enumerate}
\item $y$ codes a premouse $\N$,
\item there is a real $u$ such that $u$ codes an iteration $\T$ of $\S$ such that 
\begin{enumerate}
\item $\T$ has a last model $\N$,
\item $\pi^\T$ is defined,
\item for every limit $\a<\lh(\T)$, letting $b=[0, \a)_\T$, $\Q(b, \T\rest \a)$ exists and $\Q(b, \T\rest \a)\inseg \M_{2n}(\c(\T\rest \a))$. 
\end{enumerate}
\end{enumerate}
The complexity of the statement comes from clause 2.3, and \cite[Corollary 4.9]{PWIM} implies that it is $\Sigma^1_{2n+2}(u)$. The reason is that for any $v\in \bR$, $\bR\cap \M_{2n+2}(v)$ is the largest countable $\Sigma^1_{2n+2}(v)$ set\footnote{Another way of seeing this is just that if for every $x\in \bR$, $\M_{2n}(x)^\#$ exists then if $\W$ is a $2n$-small $\Pi^1_{2n+1}$-iterable $\d$-mouse (in the sense of \cite[]{PWIM}) then $\W$ is $\omega_1+1$-iterable. This can be shown by appealing to \cite[Lemma 3.3]{PWIM}.}. It is then not hard to see that $``y$ codes a complete iterate of $\S$" is indeed $\Sigma^1_{2n+2}(x)$.
A very similar calculation shows that the statement in clause 2 is also $\Sigma^1_{2n+2}(x)$.
\end{proof}

Recall that \rlem{stable p} says that for every $\gg<\d^1_{2n+2}$ there is a $\gg$-stable complete iterate $\N$ of $\M$. Before we state and prove Hjorth's Reflection argument, we will need the following lemma.
\begin{lemma}\label{fact we need} There is a complete $\gg$-stable iterate $\N$ of $\M$ such that if $\P={\sf{StrLe}}(\N)$ then $\P$ is a $\gg$-stable complete iterate of $\M$. 
\end{lemma}
\begin{proof} Let $\N'$ be any $\gg$-stable complete iterate of $\M$ and set $\R=\M_{2n+1}(\M, \N')$. Let $\N={\sf{StreLe}}(\R)$ and let $\P={\sf{StreLe}}(\N)$. Then both $\N$ and $\P$ are complete iterates of both $\M$ and $\N'$. Therefore, both $\N$ and $\P$ are $\gg$-stable. 
\end{proof}
\begin{lemma}[Hjorth's Reflection]\label{hjorths trick} Suppose $\gg<\d^1_{2n+2}$ and $\N$ is a $\gg$-stable complete iterate of $\M$ such that ${\sf{StrLe}}(\N)$ is also a complete $\gg$-stable iterate of $\M$. Then whenever $g\subseteq Coll(\omega, \Q_\N)$ is $\N$-generic, $B_\gg\cap \N[g]\not=\emptyset$. 
\end{lemma}
\begin{proof}
 Set $p=p^{{\sf{StrLe}}(\N)}_{\gg, 0, 3, 1}$ (see clause 9 of \rrev{gamma stable}, here we define $p^{{\sf{StrLe}}(\N)}_{\gg, 0, 3, 1}$ relative to $\mathcal{E}^{{\sf{StrLe}}(\N)}_{sm}$). Let $\mathcal{A}\in \N$ consist of quadruples $(\a', \b', q, r)$ such that
\begin{enumerate}
\item $\a'\leq \iota_\N$, $\N|\b'\models {\sf{ZFC}}+``\a'$ is the least Woodin cardinal"+``there are $2n+1$ many Woodin cardinals",
\item  $\a'$ is an $fb$-cut of $\N$, and\\\\
 letting $\K={\sf{StrLe}}(\N|\b')$ and $\mathcal{E}=\mathcal{E}^{\K}_{sm}$, there is $(q, r)\in {\sf{Ea}}^{\K}_{\a', \mathcal{E}}\times {\sf{Ea}}^{\K}_{\a', \mathcal{E}}$,
\item $r$ is compatible with $p$,
\item $(q, r)\forces \exists b( U'_{\K[{\sf{ea}}^l, {\sf{ea}}^r], d}({\sf{ea}}^r_2, {\sf{ea}}^l, b))$ where letting $(\a', \xi_1, ...\xi_{2n})$ be the Woodin cardinals of $\K$ enumerated in increasing order, $d=(\xi_1, \xi_2,..., \xi_{2n})$. 
\end{enumerate}

\begin{sublemma}\label{sub1} Suppose $x\in \bR$. Then the following are equivalent.
\begin{enumerate}
\item $x\in A_\gg$.
\item There is a complete iterate $\S$ of $\N$ such that $\T_{\N, \S}$ is below $\iota_\N$  and for some $(\a', \b', q, r)\in \pi_{\N, \S}(\mathcal{A})$,
$x\models q$ and letting $\K={\sf{StrLe}}(\S|\b')$ and $\mathcal{E}=\mathcal{E}^\K_{sm}$, $x$ is generic over ${\sf{Ea}}^{\K}_{\a', \mathcal{E}}$.
\end{enumerate}
\end{sublemma}
\begin{proof}
(\textbf{Clause 1 implies Clause 2:}) Suppose $x\in A_\gg$. Let $x_1\in {\sf{Code}}$ be such that $\gg_{x_1}=\gg$, $x_2\in B_\gg$ and $x_3$ be such that $D(x_1, x_2, x_3)$.  Let $\P=\M_{2n+1}(z, x, (x_1, x_2, x_3))$. Notice that if $u$ is a real coding $\P$ then we in fact have that $\theta_0(u, x, x_1, x_2, x_3)$ holds.

We now compare $\P$ and $\N$ to obtain $\P'$ and $\N'$ such that \\\\
(a) the least Woodin cardinals of $\P'$ and $\N'$ coincide,\\
(b) if $\k$ is the least Woodin cardinal of $\P'$ and $\N'$ then the fully backgrounded construction of $\P|\k$ coincides with the fully backgrounded construction of $\N'|\k$\footnote{The reason that the $\P'$ and $\N'$ have the same Woodin cardinal is that if $\Q$ is some $\M_{2n+1}$-like premouse over some set $a$ then the least Woodin cardinal of $\Q$ is the least Woodin cardinal of the fully backgrounded construction. If now for example the least Woodin cardinal $\eta$ of $\P'$ is strictly smaller than the least Woodin cardinal $\nu$ of $\N'$ then because $\M_{2n}({\sf{Le}}(\P'|\eta)\models ``\eta$ is a Woodin cardinal" we have that $\eta$ must be a Woodin cardinal of ${\sf{Le}}^{\N'|\nu}$, contradiction.}\\\\
Let $i: \P\rightarrow \P'$ be the iteration embedding. \\\\
\textit{Claim.}
 $ i(\a_s)\leq \iota_{\N'}$.\\\\
 \begin{proof}
 Suppose that $i(\a_s)>\iota_{\N'}$. Because the fully backgrounded constructions of $\P'$ and $\N'$ are the same and because $\iota_{\N'}$ is an $fb$-cut of $\N'$, we have that $\iota_{\N'}$ is an $fb$-cut of $\P'$.  It now follows that $\theta(k, u, x, x_1, x_2, x_3)$ holds, and therefore, $k\in w'$.
  \end{proof}
  
Now, set $\K={\sf{StrLe}}(\N'|i(\b_s))={\sf{StrLe}}(\P'|i(\b_s))$ and $\mathcal{E}=\mathcal{E}^{\K|i(\a_s)}_{sm}$. Let $(q, r)\in {\sf{Ea}}^{\K}_{i(\a_s), \mathcal{E}}\times {\sf{Ea}}^{\K}_{i(\a_s), \mathcal{E}}$ be such that
\begin{enumerate}
\item $x\models q$,
\item $(x_1, x_2, x_3)\models r$,
\item $(q, r)\forces \exists b( U'_{\K[{\sf{ea}}^l, {\sf{ea}}^r], d}({\sf{ea}}^r_2, {\sf{ea}}^l, b))$ where letting $(\a_s, \xi_1,..., \xi_{2n})$ be the Woodin cardinals of $\K$ enumerated in increasing order, $d=(\xi_1, \xi_2,..., \xi_{2n})$. 
\end{enumerate}
Assuming that $\pi_{\N, \N'}(p)$ is compatible with $r$, we get that $(i(\a_s), i(\b_s), q, r)\in \pi_{\N, \N'}(\mathcal{A})$. To finish we need to see that in fact
$\pi_{\N, \N'}(p)$ is compatible with $r$.  Notice now that because $\gg_{x_1}=\gg$, $(x_1, x_2, x_3)\models \pi_{\N, \N'}(p)$\footnote{We have that $(x_1, x_2, x_3)$ is generic over ${\sf{StrLe}}(\P')$ and since ${\sf{StrLe}}(\P')={\sf{StrLe}}(\N')$, we have that $(x_1, x_2, x_3)\models \pi_{\N, \N'}(p)$.}.
Because $(x_1, x_2, x_3)\models r$, we must have that $r$ is compatible with $\pi_{\N, \N'}(p)$. Thus, $(i(\a_s), i(\b_s), q, r)\in \pi_{\N, \N'}(\mathcal{A})$. 

Set now $\T=\T_{\N, \N'}$ and $\tau=\lh(\T)-1$. Because $\iota_\N$ is a cutpoint of $\N$, we can find $\xi<\lh(\T)$ such that 
\begin{enumerate}
\item $\M_\xi^\T|\iota_{\N'}=\N'|\iota_{\N'}$,
\item the generators of $\T\rest \xi$ are contained in $\iota_{\N'}$,
\item $\xi\in [0, \tau]_{\T}$,
\item all the extenders used in $[0, \tau)_\T$ after stage $\xi$ have critical points $>\iota_\N$. 
\end{enumerate}
Set $\S=\M_\xi^\T$. We then have that $\pi_{\N, \N'}(\mathcal{A})=\pi_{\N, \S}(\mathcal{A})$, and the pair $(\T\rest \xi+1, (i(\a_s), i(\b_s), q, r))$ witnesses clause 2 of \rlem{sub1}.\\

(\textbf{Clause 2 implies Clause 1.}) Conversely, suppose $\S$ is a complete iterate of $\N$ such that $\T_{\N, \S}$ is based on $\Q_\N$ and for some $(\a', \b', q, r)\in \pi_{\N, \S}(\mathcal{A})$,
$x\models q$ and setting $\K={\sf{StrLe}}(\S|\b')$ and $\mathcal{E}=\mathcal{E}^{\K}_{sm}$, $x$ is generic over ${\sf{Ea}}^{\K}_{\a', \mathcal{E}}$. We now further iterate $\S$ above $\a'$ to get $\S'$ such that $x$ is generic over ${\sf{Ea}}^{\S'}_{\d_{\S'}, \a'}$\footnote{This is the extender algebra that uses extenders with critical points $>\a'$.}. We then have that $x$ is actually generic over $\P=_{def}{\sf{StrLe}}(\S')$ (see Clause 2 of \rrev{s constructions}). It follows that there is a $\P[x]$-generic $(x_1, x_2, x_3)$ such that $(x_1, x_2, x_3)\models r\wedge \pi_{\N, \S'}(p)$. Therefore, we must have that $\gg_{x_1}=\gg$ and $x_2\in B_\gg$. If we now show that $U(x_2, x)$ holds then we will get that  $x\in A_\gg$.

To show that $U(x_2, x)$ holds it is enough to show that $(x_1, x_2, x_3)$ is generic over $\K[x]$ for ${\sf{Ea}}^{\K}_{\a', \mathcal{E}}$. Assuming this, we have that since
\begin{center} $(q, r)\forces \exists b( U'_{\K[{\sf{ea}}^l, {\sf{ea}}^r], d}({\sf{ea}}^r_2, {\sf{ea}}^l, b))$ \end{center} and $\K[x, (x_1, x_2, x_3)]$ is $\Sigma^1_{2n+2}$ correct, in fact $\exists b U'(x_2, x, b)$ holds in $V$. Therefore, $U(x_2, x)$ holds. 

The fact that $(x_1, x_2, x_3)$ is generic over $\K[x]$ for ${\sf{Ea}}^{\K}_{\a', \mathcal{E}}$ follows from the Claim in the proof of \cite[Theorem 2.2]{Hjorth}. To apply that Claim, set $g$ to be $(x_1, x_2, x_3)$, $\psi_0$ to be $r\wedge \pi_{\N, \S}(p)$ and $y$ to be $x$. Here we use the fact that $\a'$ is an $fb$-cut of $\S'|\iota_{\S'}$ and $\iota_{\S'}$ is an $fb$-cut of $\S'$, implying that in fact $\a'$ is an $fb$-cut of $\S'$. Thus, $\K|\a'\inseg \P$ and $\s'$ is a cardinal of $\P$.
\end{proof}
To finish the proof, notice that if $g\subseteq Coll(\omega, (\iota_\N^+)^\N)$ is generic over $\N$, $y\in \N[g]$ codes $\Q_\N$ and $\sigma$ is the formula displayed in clause 2 of \rsublem{sub1} then in fact $A_\gg\in \Sigma^1_{2n+2}(y)$ as witnessed by $\sigma$ (here we use clause 1 of \rlem{complexity}). Therefore, we can find $y'$ that is Turing reducible to $y$ and such that $u\in A_\gg \iff U(y', u)$. Hence, we have that $y'\in B_\gg\cap \N[g]$.
\end{proof}

\subsection{Removing the use of Kechris-Martin}\label{remove km}
Recall that $\Q=\M|(\iota_\M^+)^\M$ (where $\iota_\M, \zeta_\M$ were introduced in the begining of \rsec{sec:href}) and if $\S$ is a complete iterate of $\M$ then let $\Q_\S=\pi_{\M, \S}(\Q)$. We have that $\Q\inseg \M|\k_\M$ and for each $\gg$ there is a complete iterate $\S$ of $\M$ such that $\T_{\M, \S}$ is below $\iota_\M$ and if $g\subseteq Coll(\omega, \pi_{\M, \S}(\Q))$ is generic over $\S$ then $\S[g]\cap B_\gg\not=\emptyset$.  Fix some recursive enumeration $(\phi_e: e\in \omega)$ of recursive functions. 


Fix $\gg<\d^1_{2n+2}$ and let $\S$ be a complete $\gg$-stable iterate of $\M$. Let $g\subseteq Coll(\omega, \Q_\S)$ be $\S$-generic, $a_g=\{(i, j)\in \omega^2: g(i)\in g(j)\}$ and $e$ be such that if $v=\phi_e(a_g)$ then $U_v=A_\gg$. Let $\dot{\mathbb{P}}$ be a $Coll(\omega, \Q_\S)$-name for ${\sf{Ea}}^{\S[g]}_{\d_\S}$ and let  $\dot{s}\in \S$ be a $Coll(\omega, \Q_\S)$-name for $p_{\gg, 0, 3, 1}^{\S[g]}$, the $(\gg, 0, 3, 1)$-master condition in $\S[g]$ (see clause 9 of \rrev{gamma stable}). It follows from clause 12 of \rsec{pi iterability} that \\\\
(*) if $(\dot{s}, \dot{q})\in \dot{\mathbb{P}}\times \dot{\mathbb{P}}$ then $\S[g]\models (\dot{s}_g, \dot{q}_g)\forces ``D({\sf{ea}}^r_1, {\sf{ea}}^r_2,{\sf{ea}}^r_3) \rightarrow (U_v= U_{{\sf{ea}}^r_2})"$. \\\\
Let $\dot{a}$ be $Coll(\omega, \Q_\S)$-name for $a_g$. It follows from (*) that there is a condition $p\in Coll(\omega, \Q_\S)$ and $e\in \omega$ such that\\\\
(**) whenever $\dot{q}\in \dot{\mathbb{P}}$, $\S\models (p, \dot{s}, \dot{q})\forces ``D({\sf{ea}}^r_1, {\sf{ea}}^r_2,{\sf{ea}}^r_3) \rightarrow (U_{\phi_e(\dot{a})}= U_{{\sf{ea}}^r_2})"$.\\\\
 Let $( p_\gg^\S, e^\S_\gg)$ be the lexicographically least pair $(p, e)$ satisfying (**). 

Suppose now that $\S'$ is a complete iterate of $\S$. Then because $\S$ is $\gg$-stable we have that $\pi_{\S, \S'}(p_\gg^\S, e^\S_\gg)=(p_\gg^{\S'}, e^{\S'}_{\gg})$. Set then $(p_{\gg, \infty}, e_{\gg, \infty})=\pi_{\S, \infty}(p_\gg^\S, e_\gg^\S)$.

We now claim that if $\gg\not=\gg'$ then $(p_{\gg, \infty}, e_{\gg, \infty})\not =(p_{\gg', \infty}, e_{\gg', \infty})$. Suppose to the contrary that $(p_{\gg, \infty}, e_{\gg, \infty})=(p_{\gg', \infty}, e_{\gg', \infty})$. Let $\S$ be a complete iterate of $\M$ which is both $\gg$ and $\gg'$ stable. It is enough to show that $(p^{\S}_{\gg}, e_\gg^\S)\not=(p_{\gg'}^\S, e_{\gg'}^\S)$. Suppose then that $(p^{\S}_{\gg}, e_\gg^\S)=(p_{\gg'}^\S, e_{\gg'}^\S)=_{def}(p, e)$. Let $g\subseteq Coll(\omega, \Q_\S)$ be $\S$-generic such that $p\in g$ and let $v=\phi_e(a_g)$. Let $s\in {\sf{Ea}}^{\S[g]}_{\d_\S}$ be the $(\gg, 0 , 3, 1)$-master condition and let $s'\in {\sf{Ea}}^{\S[g]}_{\d_\S}$ be the $(\gg', 0, 3, 1)$-master condition.

 Because $A_\gg\not =A_{\gg'}$, we can, without losing generality, assume that there is $x\in A_\gg-A_{\gg'}$. Fix such an $x$. 

We now have that\\\\
(***) there is a condition $(r, r')\in {\sf{Ea}}^{\S[g]}_{\d_\S}\times {\sf{Ea}}^{\S[g]}_{\d_\S}$ such that $(r, r')$ extends $(s, s')$ and 
\begin{center}
$\S[g]\models (r, r')\forces D({\sf{ea}}^l_1, {\sf{ea}}^l_2,{\sf{ea}}^l_3)\wedge D({\sf{ea}}^r_1, {\sf{ea}}^r_2,{\sf{ea}}^r_3)$.
\end{center}
(***) can be shown by iterating $\S[g]$ to make two triples $(y_1, y_2, y_3)$ and $(y_1', y_2', y_3')$ generic where the triples are chosen so that
\begin{enumerate}
\item $D(y_1, y_2, y_3)$ and $D(y'_1, y'_2, y'_3)$,
\item $\gg_{y_1}=\gg$ and $\gg_{y_1'}=\gg'$. 
\end{enumerate}
Then elementarity would imply (***).  Fix such a condition $(r, r')$. 

Suppose $\R$ is a complete iterate of $\S[g]$ such that $x$ is generic for ${\sf{Ea}}^{\R[g]}_{\d_\R}$ over $\R[g]$.
Let $(x_1, x_2, x_3)$ be any $\R[g, x]$-generic for ${\sf{Ea}}^{\R[g]}_{\d_\R}$ with the property that $(x_1, x_2, x_3)\models \pi_{\S[g], \R[g]}(r)$ and let $(x_1', x_2', x_3')$ be any $\R[g, x]$-generic for ${\sf{Ea}}^{\R[g]}_{\d_\R}$ with the property that $(x_1', x_2', x_3')\models  \pi_{\S[g], \R[g]}(r')$. It follows from (***) that \begin{enumerate}
\item $D(x_1, x_2, x_3)$ and $D(x_1', x_2', x_3')$,
\item $\gg_{x_1}=\gg$ and $\gg_{x_1'}=\gg'$. 
\end{enumerate} 
We now have that (**) implies that\\\\
(1) if $q\in {\sf{Ea}}^{\R[g]}_{\d_\R}$ then $\R[g, (x_1, x_2, x_3)]\models q\forces ``U_v=U_{x_2}"$, \\
(2) if $q\in {\sf{Ea}}^{\R[g]}_{\d_\R}$ then $\R[g, (x_1', x_2', x_3')]\models q\forces ``U_v=U_{x_2'}"$, \\\\
Therefore, we get that\\\\
(3) $\R[g, (x_1, x_2, x_3)][x]\models ``U_v=U_{x_2}"$ and $\R[g, (x_1', x_2', x_3')][x]\models  ``U_v=U_{x_2'}"$.\\\\
We now use (3) and clause 12 of \rsec{pi iterability} to make the following implications showing that $x\in A_{\gg'}$ (recall that we chose $x\in A_\gg-A_{\gg'}$):
\begin{align*}
U_{x_2}=A_\gg &\rightarrow x\in (U_{x_2})^{\R[g, x, (x_1, x_2, x_3)]}\\
&\rightarrow  x\in (U_v)^{\R[g, x, (x_1, x_2, x_3)]}\\
& \rightarrow x\in U_v\\
& \rightarrow x\in (U_v)^{\R[g, x, (x_1', x_2', x_3')]}\\
& \rightarrow x \in (U_{x_2'})^{\R[g, x, (x_1', x_2', x_3')]}\\
&\rightarrow x\in U_{x_2'}\\
&\rightarrow x\in A_{\gg'}
\end{align*}
Since we now have that the function $f(\gg)=(p_{\gg, \infty}, e_{\gg, \infty})$ is an injection, $\card{\pi_{\M, \infty}(\Q)}=\d^1_{2n+2}$.

 However, because ${\sf{Ord}}\cap \Q$ is a cutpoint of $\M$, $\pi_{\M, \infty}(\Q)$ is the direct limit of all countable iterates of $\Q$, which we denoted by $\M_\infty(\Q)$ (see clause 3 of \rrev{ds}). It then follows from clause 2 of \rlem{complexity} that $\card{\M_\infty(\Q)}<\delta^1_{2n+2}$. Indeed, consider the set $E=\{ u\in \bR: u$ codes a pair $(\R_u, \iota_u)$ such that  $\R_u$ is a complete iterate of $\Q$ and $\iota_u\in {\sf{Ord}}\cap \R_u\}$ and for $u, u'\in E$, set $u\leq^*_0u'$ if and only if letting $\R$ be the complete common iterate of $\R_u$ and $\R_{u'}$, $\pi_{\R_u, \R}(\iota_u)\leq \pi_{\R_{u'}, \R}(\iota_{u'})$. It now follows from clause 2 of \rlem{complexity} that $\leq^*_0$ is $\Sigma^1_{2n+2}(u_0)$ where $u_0$ codes $\Q$. Hence, the ordinal length of $\leq^*_)$ is $<\d^1_{2n+2}$\footnote{Here we are using Kunen-Martin theorem, see \cite[Theorem 2.6]{Jackson}.}, and therefore $\card{\M_\infty(\Q)}<\d^1_{2n+2}$.

\bibliographystyle{plain}
\bibliography{HjorthsRef}

\begin{thebibliography}{10}

\bibitem{ANDHJONEE}
Alessandro Andretta, Greg Hjorth, and Itay Neeman.
\newblock Effective cardinals of boldface pointclasses.
\newblock {\em J. Math. Log.}, 7(1):35--82, 2007.

\bibitem{CaicedoKetchersid}
Andr\'{e}s~Eduardo Caicedo and Richard Ketchersid.
\newblock A trichotomy theorem in natural models of ${A}{D}^+$.
\newblock In {\em Set theory and its applications}, volume 533 of {\em Contemp. Math.}, pages 227--258. Amer. Math. Soc., Providence, RI, 2011.

\bibitem{ChanDet}
William Chan.
\newblock An introduction to combinatorics of determinacy.
\newblock In {\em Trends in set theory}, volume 752 of {\em Contemp. Math.}, pages 21--75. Amer. Math. Soc., Providence, RI, [2020] \copyright 2020.

\bibitem{Marks}
Adam Day and Andrew Marks.
\newblock The decomposability conjecture, 2021, available at https://www.youtube.com/watch?v=Nio3PkWBRz4.

\bibitem{Harrington}
Leo Harrington.
\newblock Analytic determinacy and {$0^{\sharp }$}.
\newblock {\em J. Symbolic Logic}, 43(4):685--693, 1978.

\bibitem{Dichotomy}
Greg Hjorth.
\newblock A dichotomy for the definable universe.
\newblock {\em J. Symbolic Logic}, 60(4):1199--1207, 1995.

\bibitem{Hjorth}
Greg Hjorth.
\newblock Two applications of inner model theory to the study of {$\utilde\Sigma^1_2$} sets.
\newblock {\em Bull. Symbolic Logic}, 2(1):94--107, 1996.

\bibitem{BoundHjorth}
Greg Hjorth.
\newblock A boundedness lemma for iterations.
\newblock {\em J. Symbolic Logic}, 66(3):1058--1072, 2001.

\bibitem{Hjorth_2002}
Greg Hjorth.
\newblock Cardinalities in the projective hierarchy.
\newblock {\em Journal of Symbolic Logic}, 67(4):1351–1372, 2002.

\bibitem{ICMJackson}
Stephen Jackson.
\newblock Towards a theory of definable sets.
\newblock In {\em Proceedings of the {I}nternational {C}ongress of {M}athematicians---{R}io de {J}aneiro 2018. {V}ol. {II}. {I}nvited lectures}, pages 25--44. World Sci. Publ., Hackensack, NJ, 2018.

\bibitem{JacksonPart}
Steve Jackson.
\newblock Partition properties and well-ordered sequences.
\newblock {\em Ann. Pure Appl. Logic}, 48(1):81--101, 1990.

\bibitem{delta15}
Steve Jackson.
\newblock A computation of {$\delta^1_5$}.
\newblock {\em Mem. Amer. Math. Soc.}, 140(670):viii+94, 1999.

\bibitem{Jackson}
Steve Jackson.
\newblock Structural consequences of {AD}.
\newblock In {\em Handbook of set theory. {V}ols. 1, 2, 3}, pages 1753--1876. Springer, Dordrecht, 2010.

\bibitem{Jech}
Thomas Jech.
\newblock {\em Set theory}.
\newblock Springer Monographs in Mathematics. Springer-Verlag, Berlin, 2003.
\newblock The third millennium edition, revised and expanded.

\bibitem{JSSS}
Ronald Jensen, Ernest Schimmerling, Ralf Schindler, and John Steel.
\newblock Stacking mice.
\newblock {\em J. Symbolic Logic}, 74(1):315--335, 2009.

\bibitem{Kanamori}
Akihiro Kanamori.
\newblock {\em The higher infinite}.
\newblock Perspectives in Mathematical Logic. Springer-Verlag, Berlin, 1994.
\newblock Large cardinals in set theory from their beginnings.

\bibitem{Kechris}
Alexander~S. Kechris.
\newblock On transfinite sequences of projective sets with an application to $\utilde{\Sigma}^1_2$ equivalence relations.
\newblock In {\em Logic {C}olloquium '77 ({P}roc. {C}onf., {W}roclaw, 1977)}, volume~96 of {\em Stud. Logic Foundations Math.}, pages 155--160. North-Holland, Amsterdam-New York, 1978.

\bibitem{ADDC}
Alexander~S. Kechris.
\newblock The axiom of determinacy implies dependent choices in {L}({{\({\mathbb{R}})\)}}.
\newblock {\em J. Symb. Log.}, 49:161--173, 1984.

\bibitem{KM}
Alexander~S. Kechris and Donald~A. Martin.
\newblock On the theory of {$\Pi^1_3$} sets of reals, {II}.
\newblock In {\em Ordinal definability and recursion theory: {T}he {C}abal {S}eminar. {V}ol. {III}}, volume~43 of {\em Lect. Notes Log.}, pages 200--219. Assoc. Symbol. Logic, Ithaca, NY, 2016.

\bibitem{Qtheory}
Alexander~S. Kechris, Donald~A. Martin, and Robert~M. Solovay.
\newblock Introduction to {$Q$}-theory.
\newblock In {\em Cabal seminar 79--81}, volume 1019 of {\em Lecture Notes in Math.}, pages 199--282. Springer, Berlin, 1983.

\bibitem{ADPlusBook}
Paul Larson.
\newblock {E}xtensions of the {A}xiom of {D}eterminacy.
\newblock {\em To appear, available at http://www.users.miamioh.edu/larsonpb}.

\bibitem{IT}
Donald~A. Martin and John~R. Steel.
\newblock Iteration trees.
\newblock {\em J. Amer. Math. Soc.}, 7(1):1--73, 1994.

\bibitem{FSIT}
William~J. Mitchell and John~R. Steel.
\newblock {\em Fine structure and iteration trees}, volume~3 of {\em Lecture Notes in Logic}.
\newblock Springer-Verlag, Berlin, 1994.

\bibitem{Sandra}
Sandra Mueller and Grigor Sargsyan.
\newblock {H}{O}{D} in inner models with woodin cardinals.
\newblock {\em To appear in the Journal of Symbolic Logic, available at http://www.grigorsargis.net/papers.html}.

\bibitem{PDMice}
Sandra M\"{u}ller, Ralf Schindler, and W.~Hugh Woodin.
\newblock Mice with finitely many {W}oodin cardinals from optimal determinacy hypotheses.
\newblock {\em J. Math. Log.}, 20(suppl. 1):1950013, 118, 2020.

\bibitem{IMU}
Itay Neeman.
\newblock Inner models and ultrafilters in {$L(\Bbb R)$}.
\newblock {\em Bull. Symbolic Logic}, 13(1):31--53, 2007.

\bibitem{NegIdeal}
Grigor Sargsyan.
\newblock Negative results on precipitous ideals on $\omega_1$.
\newblock {\em To appear in the Journal of Symbolic Logic, available at http://www.grigorsargis.net/papers.html}.

\bibitem{BSL}
Grigor Sargsyan.
\newblock A short tale of hybrid mice, to appear at the {B}ulletin of {S}ymbolic {L}ogic.

\bibitem{SargPre}
Grigor Sargsyan.
\newblock On the prewellorderings associated with the directed systems of mice.
\newblock {\em J. Symbolic Logic}, 78(3):735--763, 2013.

\bibitem{ATHM}
Grigor Sargsyan.
\newblock {\em Hod mice and the mouse set conjecture}, volume 236 of {\em Memoirs of the {A}merican {M}athematical {S}ociety}.
\newblock American Mathematical Society, 2014.

\bibitem{CoarseAD}
Grigor Sargsyan.
\newblock {$AD_{\Bbb R}$} implies that all sets of reals are {$\Theta$} universally {B}aire.
\newblock {\em Arch. Math. Logic}, 60(1-2):1--15, 2021.

\bibitem{Varsovian}
Grigor Sargsyan and Ralf Schindler.
\newblock Varsovian models {I}.
\newblock {\em J. Symb. Log.}, 83(2):496--528, 2018.

\bibitem{Selfiterability}
Ralf Schindler and John Steel.
\newblock The self-iterability of {$L[E]$}.
\newblock {\em J. Symbolic Logic}, 74(3):751--779, 2009.

\bibitem{trang2013}
F.~Schlutzenberg and N.~Trang.
\newblock Scales in hybrid mice over $\mathbb{R}$.
\newblock {\em arXiv preprint arXiv:1210.7258v4}, 2014.

\bibitem{farm2}
Farmer Schlutzenberg.
\newblock Full normalization for transfinite stacks.
\newblock {\em arXiv preprint arXiv:2102.03359}, 2014.

\bibitem{MeasuresInMice}
Farmer~Salamander Schlutzenberg.
\newblock {\em Measures in mice}.
\newblock ProQuest LLC, Ann Arbor, MI, 2007.
\newblock Thesis (Ph.D.)--University of California, Berkeley.

\bibitem{PWIM}
John Steel.
\newblock Projectively well-ordered inner models.
\newblock {\em Ann. Pure Appl. Logic}, 74(1):77--104, 1995.

\bibitem{Steel1995}
John~R. Steel.
\newblock {${\rm HOD}\sp {{\rm{L}}(\Bbb R)}$} is a core model below {$\Theta$}.
\newblock {\em Bull. Symbolic Logic}, 1(1):75--84, 1995.

\bibitem{OIMT}
John~R. Steel.
\newblock An outline of inner model theory.
\newblock In {\em Handbook of set theory. {V}ols. 1, 2, 3}, pages 1595--1684. Springer, Dordrecht, 2010.

\bibitem{HODCoreModel}
John~R. Steel and W.~Hugh Woodin.
\newblock {{HOD}} as a core model.
\newblock In {\em Ordinal definability and recursion theory: {T}he {C}abal {S}eminar. {V}ol. {III}}, volume~43 of {\em Lect. Notes Log.}, pages 257--345. Assoc. Symbol. Logic, Ithaca, NY, 2016.

\bibitem{OpenProblems}
AIM workshop on recent advances in core~model theory.
\newblock Open problems.
\newblock {\em available at http://www.math.cmu.edu/users/eschimme/AIM/problems.pdf}.

\end{thebibliography}
\end{document}